\documentclass[12pt]{amsart}
\usepackage{fancybox}
\usepackage{shuffle}
\usepackage{amsmath}
\usepackage{amsfonts}
\usepackage{amssymb}
\usepackage[all]{xy}           
\usepackage{bm}
\usepackage{bbding}
\usepackage{txfonts}
\usepackage{amscd}
\usepackage{appendix}

\usepackage{xspace}
\usepackage[shortlabels]{enumitem}
\usepackage{ifpdf}

\ifpdf
  \usepackage[colorlinks,final,backref=page,hyperindex]{hyperref}
\else
  \usepackage[colorlinks,final,backref=page,hyperindex,hypertex]{hyperref}
\fi
\usepackage{tikz}
\usepackage[active]{srcltx}

\makeatletter

\newtheorem{thm}{Theorem}[section]
\newtheorem{lem}[thm]{Lemma}
\newtheorem{cor}[thm]{Corollary}
\newtheorem{pro}[thm]{Proposition}
\newtheorem{ex}[thm]{Example}
\theoremstyle{definition}
\newtheorem{rmk}[thm]{Remark}
\newtheorem{defi}[thm]{Definition}

\newcommand{\nc}{\newcommand}
\newcommand{\delete}[1]{}

\nc{\mlabel}[1]{\label{#1}}  
\nc{\mcite}[1]{\cite{#1}}  
\nc{\mref}[1]{\ref{#1}}  
\nc{\mbibitem}[1]{\bibitem{#1}} 

\delete{
\nc{\mlabel}[1]{\label{#1}{\hfill \hspace{1cm}{\bf{{\ }\hfill(#1)}}}}
\nc{\mcite}[1]{\cite{#1}{{\em{{\ }(#1)}}}}  
\nc{\mref}[1]{\ref{#1}{{\em{{\ }(#1)}}}}  
\nc{\mbibitem}[1]{\bibitem[\em #1]{#1}} 
}

\newcommand {\emptycomment}[1]{}

\nc{\oprn}{\theta}
\nc{\Oprn}{\Theta}

\nc{\calo}{\mathcal{O}}
\nc{\oop}{$\mathcal{O}$-operator\xspace}
\nc{\oops}{$\mathcal{O}$-operators\xspace}
\nc{\mrho}{{\bm{\varrho}}}
\nc{\emk}{\mathbf{K}}
\nc{\invlim}{\displaystyle{\lim_{\longleftarrow}}\,}
\nc{\ot}{\otimes}

\newcommand{\lon }{\,\rightarrow\,}
\newcommand{\be }{\begin{equation}}
\newcommand{\ee }{\end{equation}}

\newcommand{\g}{\mathfrak g}
\newcommand{\h}{\mathfrak h}

\newcommand{\huaB}{\mathcal{B}}

\newcommand{\huaG}{\mathcal{G}}

\newcommand{\huaC}{{\mathcal{C}}}
\newcommand{\huaD}{\mathcal{D}}

\newcommand{\huaH}{\mathcal{H}}

\newcommand{\huaO}{{\mathcal{O}}}

\newcommand{\huaZ}{\mathcal{Z}}

\newcommand{\frkg}{\mathfrak g}

\newcommand{\frkL}{\mathfrak L}

\newcommand{\frkR}{\mathfrak R}

\newcommand{\frkX}{\mathfrak X}

\newcommand{\pair}[1]{\left\langle #1\right\rangle}

\newcommand{\Courant}[1]{\left\llbracket  #1\right\rrbracket }

\newcommand{\Id}{{\rm{Id}}}

\newcommand{\br}[1]{   [ \cdot,    \cdot  ]   }

\newcommand{\Hom}{\mathrm{Hom}}

\newcommand{\LP}{$\mathsf{LieRep}$~}
\newcommand{\LYP}{$\mathsf{LieYRep}$~}

\newcommand{\gl}{\mathfrak {gl}}

\newcommand{\ad}{\mathrm{ad}}

\nc{\CV}{\mathbf{C}}

\newcommand{\LYA}{Lie-Yamaguti algebra}
\newcommand{\ltp}{\Courant{\cdot,\cdot,\cdot}}

\begin{document}

\title{Cohomology and Relative Rota-Baxter-Nijenhuis structures on \LYP pairs}

\author{Jia Zhao}
\address{Jia Zhao, School of Sciences, Nantong University, Nantong 226019, Jiangsu, China}
\email{zhaojia@ntu.edu.cn}

\author{Yu Qiao*}
\address{Yu Qiao (corresponding author), School of Mathematics and Statistics, Shaanxi Normal University, Xi'an 710119, Shaanxi, China}
\email{yqiao@snnu.edu.cn}

\date{\today}

\begin{abstract}
A \LYP pair consists of a Lie-Yamaguti algebra and its representation.
In this paper, we establish the cohomology theory of \LYP pairs and characterize their linear deformations by the second cohomology group.
Then we introduce the notion of relative Rota-Baxter-Nijenhuis structures on \LYP pairs, investigate their properties, and prove that a relative Rota-Baxter-Nijenhuis structure gives rise to a pair of compatible relative Rota-Baxter operators under a certain condition.
Finally, we show the equivalence between $r$-matrix-Nijenhuis structures and Rota-Baxter-Nijenhuis structures on \LYA s.
\end{abstract}

\thanks{{\em Mathematics Subject Classification} (2020): Primary 17B38; Secondary 17B60,17A99}

\keywords{Lie-Yamaguti algebra, cohomology, relative Rota-Baxter-Nijenhuis structure, $r$-matrix-Nijenhuis structure}

\maketitle

\vspace{-1.1cm}

\tableofcontents

\allowdisplaybreaks

 \section{Introduction}

Poisson-Nijenhuis structures appeared in the work of Magri and Morosi \cite{Magri} in studying completely integrable systems. Such objects can also be found in \cite{KoMa,KoRu}. Later, Ravanpak, Rezaei-Aghdam, and Haghighatdoost exploited an $r-n$ structure on a Lie algebra, which is the infinitesimal right-invariant Poisson-Nijenhuis structure on the Lie group $G$ integrating the Lie algebra $\g$ \cite{RRH}. Namely, an $r-n$ structure on a Lie algebra $\g$ is a pair $(\pi,N)$, where $\pi\in \wedge^2\g$ is a classical $r$-matrix, and $N$ is a Nijenhuis operator on $\g$ such that some conditions are satisfied. Motivated by their works, Liu, Sheng and their collaborators studied Poisson-Nijenhuis structures on Lie algebras, associative algebras and pre-Lie algebras in \cite{HuLiuSheng,LW,LiuBaiSheng}, where they call this object the Kupershmidt-Nijenhuis structures or ${\huaO}N$-structures. In particular, they named an $r-n$ structure as an $r$-matrix-Nijenhuis structure and investigated its relation with the Rota-Baxter-Nijenhuis structures on a \LP pair in \cite{HuLiuSheng}. Based on their works, the first author, Liu, and Sheng explored the Poisson-Nijenhuis structures on $3$-$\mathsf{LieRep}$ pairs, in which Poisson-Nijenhuis structures were called relative Rota-Baxter-Nijenhuis structures \cite{Zhao-Liu}.

The notion of \LYA s, a generalization of Lie algebras and Lie triple systems, can be dated back to Nomizu's work on invariant affine connections on homogeneous spaces in 1950's (\cite{Nomizu}) and Yamaguti's work on general Lie triple systems and Lie triple algebras (\cite{Yamaguti1}). Then Yamaguti introduced its representation and established its cohomology theory in \cite{Yamaguti2,Yamaguti3} during 1950's to 1960's. Later until 21st century, Kinyon and Weinstein named this object as a \LYA ~when studying Courant algebroids in \cite{Weinstein}. \LYA s have attracted much attention in recent years. For instance, Benito, Draper, and Elduque investigated \LYA s related to simple Lie algebras of type $G_2$ \cite{B.D.E}. Afterwards, Benito, Elduque, and Mart$\acute{i}$n-Herce explored irreducible \LYA s in \cite{B.E.M1,B.E.M2}. More recently, Benito, Bremmer, and Madariaga examined orthogonal \LYA s in \cite{B.B.M}.

Baxter first introduced the notion of Rota-Baxter operators on associative algebras when studying fluctuation theory in 1960's \cite{Ba}.
A more generalized concept -- relative Rota-Baxter operators (also called $\huaO$-operators or Kupershmidt operators) on Lie algebras, was introduced by Kupershmidt in the study of classical Yang-Baxter equation, which has lots of applications in physics and mathematical physics \cite{Kupershmidt}. Kupershmidt found that a relative Rota-Baxter operator with respect to the coadjoint representation is a solution to the classical Yang-Baxter equation.
Note that a relative Rota-Baxter operator is associated to an arbitrary representation, whereas a Rota-Baxter operator is related with the adjoint representation. One can see \cite{AMM,Bai-Bellier-Guo-Ni,B.G.S,BaiRGuo} for Rota-Baxter operators on other type of algebras and see \cite{Gub} for more details about Rota-Baxter algebras.
 Note that a relative Rota-Baxter operator on a Lie algebra (an associative algebra) is a solution to the Maurer-Cartan equation of a certain graded Lie algebra (\cite{TBGS,Uchino1}), whereas a Poisson structure is also a solution to a Maurer-Cartan equation of a graded Lie algebra whose structure are precisely the Schouten-Nijenhuis bracket of multi-vector fields. Thus a relative Rota-Baxter operator on a Lie algebra (an associative algebra) are analogues of a Poisson structure.
Another kind of crucial operators--a Nijenhuis operator on a Lie algebra appeared in the study of integrable systems by Magri, Gelfand, and Dorfman \cite{Dorfman}, and then Fuchssteiner and Fokas rediscovered it and named this kind of operators as hereditary operators in \cite{FF}. In deformation theory of Lie algebras, a Nijenhuis operator can be obtained via a trivial deformation. One can see \cite{Dorfman,Liu-Sheng,NR} for more details about Nijenhuis operators on Lie algebras or even $n$-Lie algebras.

Recently, the first author and Sheng explored linear deformations, Nijenhuis operators, and relative Rota-Baxter operators on \LYA s in \cite{ShengZhao} and \cite{SZ1} respectively.
We studied cohomology and deformations of relative Rota-Baxter operators on \LYA s in \cite{ZhaoQiao}, and established bialgebra theory of \LYA s in \cite{ZQ2}. Motivated by the virtual of Poisson-Nijenhuis structures and our former works on \LYA s, we shall investigate relative Rota-Baxter-Nijenhuis structures on a \LYA ~with a representation. By a \LYP pair, denoted by $\Big((\g,\br,,\ltp),(V;\rho,\mu)\Big)$, we mean a \LYA ~$(\g,\br,,\ltp)$ with a representation $(V;\rho,\mu)$. Note that we still call a Poisson-Nijenhuis structure on a \LYP pair a relative Rota-Baxter-Nijenhuis structure in the present paper.

\medskip
{\bf Outline of the paper.} After recalling some basic definitions in Section 2, we establish the cohomology theory of \LYP pairs and utilize it to characterize linear deformations in Section 3. The corresponding complex is exactly a subcomplex of that of the associated semidirect product \LYA ~with coefficients in the adjoint representation. Then we investigate the linear deformations of \LYP pairs and introduce the notion of Nijenhuis structures $(N,S)$, where $N\in \gl(\g)$ and $S\in \gl(V)$.  A Nijenhuis structure gives rise to a Nijenhuis operator $N+S$ on the semidirect \LYA ~$(\g\oplus V,\br_{_\ltimes},\ltp_\ltimes)$ and
a trivial deformation of $(\g, \rho, \mu)$.

In Section 3, the notion of relative Rota-Baxter-Nijenhuis structures on \LYP pairs is introduced and some properties are investigated. Given linear maps $N\in \gl(\g)$ and $S\in \gl(V)$, and let $T:V\longrightarrow\g$ be a relative Rota-Baxter operator on a \LYA ~$(\g,\br,,\ltp)$ with respect to $(V;\rho,\mu)$. Then there exists a \LYA ~structure $(\br^{^T},\ltp^T)$ on $V$, which is called the sub-adjacent \LYA, and is denoted by $V^T$. Consequently, we obtain the deformed brackets $(\br^{^T_S},\ltp^T_S)$; meanwhile, we introduce linear maps $\hat\rho:\g\to \gl(V)$ and $\hat\mu:\otimes^2\g\to \gl(V)$ associated with linear maps $N$ and $S$ to obtain the brackets $(\br^{^T_{\hat\rho}},\ltp^T_{\hat\mu})$.
It is {\em not} true that $(\br^{^T_S},\ltp^T_S)$ or $(\br^{^T_{\hat\rho}},\ltp^T_{\hat\mu})$ forms a \LYA ~structure on $V$. In order to tackle this problem, we introduce the main object of this paper: {\em relative Rota-Baxter-Nijenhuis structures}. We find that deformed brackets $(\br^{^T_S},\ltp^T_S)$ and $(\br^{^T_{\hat\rho}},\ltp^T_{\hat\mu})$ are equal under the condition that the triple $(T, S, N)$ is a relative Rota-Baxter-Nijenhuis structure. Besides, the linear map $S$ is also a Nijenhuis operator on the sub-adjacent \LYA ~$(V,\br^{^T},\ltp^T)$, hence both $(\br^{^T_S},\ltp^T_S)$ and $(\br^{^T_{\hat\rho}},\ltp^T_{\hat\mu})$ are \LYA ~structures on $V$. Furthermore, we show that the composition $N\circ T$ is a relative Rota-Baxter operator on the \LYA ~$(\g,\br,,\ltp)$ with respect to $(V;\rho,\mu)$. These results are summarized in the following commutative diagram:
\[
 \xymatrix{
 (V,[\cdot,\cdot]_S^T,\Courant{\cdot,\cdot,\cdot}_S^T)\ar[d]^{S}\ar[dr]^{N \circ T} \ar[r]^{T} & (\g,[\cdot,\cdot]_N,\Courant{\cdot,\cdot,\cdot}_N)\ar[d]^{N}    \\
(V,[\cdot,\cdot]^T,\Courant{\cdot,\cdot,\cdot}^T) \ar[r]^{T} &(\g,\br,\ltp),
}
\]
where each arrows is a \LYA ~homomorphism, i.e., each $T$ is a relative Rota-Baxter operator, and both $S$ and $N$ are Nijenhuis operators.

Unlike those on Lie algebras or associative algebras, a relative Rota-Baxter-Nijenhuis structure on a \LYP pair does {\em not} give rise to a hierarchy of relative Rota-Baxter operators and what is worse,  it does {\em not} guarantee that $T$ and $N\circ T$ are compatible even though $N\circ T$ is a relative Rota-Baxter operator. 
They, however, are compatible if we impose an extra condition on a relative Rota-Baxter-Nijenhuis structure.
This phenomenon is similar to the context of $3$-$\mathsf{LieRep}$ pairs (\cite{Zhao-Liu}) and demonstrates that properties of a relative Rota-Baxter-Nijenhuis structure on ternary operations are distinguished from those on binary operations.

Finally in Section 5, we introduce the notion of relative Rota-Baxter-dual Nijenhuis structures by using dual-Nijenhuis structures, which is dual to that of Nijenhuis structures.
In fact, an $r$-matrix-Nijenhuis structure is indeed a special relative Rota-Baxter-dual Nijenhuis structure with respect to the coadjoint representation and  we show that $r$-matrix-Nijenhuis structures and Rota-Baxter-Nijenhuis structures are equivalent under certain conditions.
\emptycomment{
The paper is structured as follows.
In Section 2, we recall some notions such as \LYA s, representations, and cohomology theory.
In Section 3, we construct the cohomology theory of \LYP pairs using Yamaguti cohomology and utilize the cohomology theory to classify linear deformations. 
In Section 4, we introduce the notion of relative Rota-Baxter-Nijenhuis structures on \LYP pairs and study several properties.
Finally in Section 6, we introduce the concept of $r$-matrix-Nijenhuis structures and establish an equivalence between $r$-matrix-Nijenhuis structures and Rota-Baxter-Nijenhuis structures on \LYA s.}

\smallskip
{\bf Acknowledgements}: Qiao was partially supported by NSFC grant 11971282.

\smallskip
\section{Preliminaries}
All vector spaces in the article are assumed to be over a field $\mathbb{K}$ of characteristic $0$ and finite-dimensional.
In this section, we recall some basic notions such as Lie-Yamaguti algebras, representations and their cohomology theory.
\begin{defi}(\cite{Weinstein})\label{LY}
A {\bf Lie-Yamaguti algebra} is a vector space $\g$ equipped with a bilinear operation $[\cdot,\cdot]:\wedge^2  \mathfrak{g} \to \mathfrak{g} $ and a trilinear operation $\Courant{\cdot,\cdot,\cdot}:\wedge^2\g \otimes  \mathfrak{g} \to \mathfrak{g} $, such that the following conditions hold
\begin{eqnarray*}
~ &&\label{LY1}[[x,y],z]+[[y,z],x]+[[z,x],y]+\Courant{x,y,z}+\Courant{y,z,x}+\Courant{z,x,y}=0,\\
~ &&\Courant{[x,y],z,w}+\Courant{[y,z],x,w}+\Courant{[z,x],y,w}=0,\\
~ &&\label{LY3}\Courant{x,y,[z,w]}=[\Courant{x,y,z},w]+[z,\Courant{x,y,w}],\\
~ &&\Courant{x,y,\Courant{z,w,t}}=\Courant{\Courant{x,y,z},w,t}+\Courant{z,\Courant{x,y,w},t}+\Courant{z,w,\Courant{x,y,t}},\label{fundamental}
\end{eqnarray*}
for all $x,y,z,w,t \in \g$. We denote a Lie-Yamaguti algebra by $(\g,\br,,\ltp)$.
\end{defi}

Note that a Lie-Yamaguti algebra with $[x,y]=0$ for all $x,y\in \g$ reduces to a Lie triple system, while a Lie-Yamaguti algebra with $\Courant{x,y,z}=0$ for all $x,y,z\in \g$ reduces to a Lie algebra. 
\emptycomment{
\begin{ex}
 Let $(\frkg,[\cdot,\cdot])$ be a Lie algebra. We define $\Courant{\cdot,\cdot,\cdot
 }:\wedge^2\g\otimes \g\lon \g$ to be  $$\Courant{x,y,z}:=[[x,y],z],\quad \forall x,y, z \in \mathfrak{g}.$$  Then $(\g,[\cdot,\cdot],\Courant{\cdot,\cdot,\cdot})$ forms a Lie-Yamaguti algebra.
\end{ex}
}
The following example is taken from \cite{Nomizu}.
\begin{ex}
Let $M$ be a closed manifold \footnote{a smooth compact manifold without boundary} with an affine connection, and denote by $\frkX(M)$ the set of vector fields on $M$. For all $x,y,z\in \frkX(M)$, set
\begin{eqnarray*}
[x,y]&=&-T(x,y),\\
\Courant{x,y,z}&=&-R(x,y)z,
\end{eqnarray*}
where $T$ and $R$ are torsion tensor and curvature tensor respectively. It turns out that the triple
$ (\frkX(M),[\cdot,\cdot],\Courant{\cdot,\cdot,\cdot})$ forms an (infinite-dimensional) \LYA.
\end{ex}

\begin{defi}\cite{Takahashi}
Let $(\g,[\cdot,\cdot]_{\g},\Courant{\cdot,\cdot,\cdot}_{\g})$ and $(\h,[\cdot,\cdot]_{\h},\Courant{\cdot,\cdot,\cdot}_{\h})$ be two Lie-Yamaguti algebras. A {\bf homomorphism} from $(\g,[\cdot,\cdot]_{\g},\Courant{\cdot,\cdot,\cdot}_{\g})$ to $(\h,[\cdot,\cdot]_{\h},\Courant{\cdot,\cdot,\cdot}_{\h})$ is a linear map $\phi:\g \to \h$ preserving the Lie-Yamaguti algebra structure, i.e., for all $x,y,z\in \g,$ the following equalities hold
\begin{eqnarray*}
\phi([x,y]_{\g})&=&[\phi(x),\phi(y)]_{\h},\\
~ \phi(\Courant{x,y,z}_{\g})&=&\Courant{\phi(x),\phi(y),\phi(z)}_{\h}.
\end{eqnarray*}
\end{defi}

The notion of representations of Lie-Yamaguti algebras was given in \cite{Yamaguti2}.
\begin{defi}\label{defi:representation}
Let $(\g,[\cdot,\cdot],\Courant{\cdot,\cdot,\cdot})$ be a Lie-Yamaguti algebra. A {\bf representation} of $\g$ is a vector space $V$ endowed with a linear map $\rho:\g \to \gl(V)$ and a bilinear map $\mu:\otimes^2 \g \to \gl(V)$, which satisfies the following conditions for all $x,y,z,w \in \g$,
\begin{eqnarray*}
~&&\label{RLYb}\mu([x,y],z)-\mu(x,z)\rho(y)+\mu(y,z)\rho(x)=0,\\
~&&\label{RLYd}\mu(x,[y,z])-\rho(y)\mu(x,z)+\rho(z)\mu(x,y)=0,\\
~&&\label{RLYe}\rho(\Courant{x,y,z})=[D_{\rho,\mu}(x,y),\rho(z)],\\
~&&\label{RYT4}\mu(z,w)\mu(x,y)-\mu(y,w)\mu(x,z)-\mu(x,\Courant{y,z,w})+D_{\rho,\mu}(y,z)\mu(x,w)=0,\\
~&&\label{RLY5}\mu(\Courant{x,y,z},w)+\mu(z,\Courant{x,y,w})=[D_{\rho,\mu}(x,y),\mu(z,w)],
\end{eqnarray*}
where $D_{\rho,\mu}$ is given by
\begin{eqnarray}
~ &&\label{rep} D_{\rho,\mu}(x,y)=\mu(y,x)-\mu(x,y)+[\rho(x),\rho(y)]-\rho([x,y]),\quad \forall x,y\in \g.
\end{eqnarray}
It is easy to see that $D_{\rho,\mu}$ is skew-symmetric. We denote a representation of $\g$ by $(V;\rho,\mu)$. In the sequel, we write $D_{\rho,\mu}$ as $D$ for short without confusion.
\end{defi}

The notion of representations of \LYA s is also a natural generalization of Lie algebras and Lie triple systems.
\emptycomment{
\begin{rmk}\label{rmk:rep}
Let $(\g,[\cdot,\cdot],\Courant{\cdot,\cdot,\cdot})$ be a Lie-Yamaguti algebra and $(V;\rho,\mu)$ its representation. If $\rho=0$ and the Lie-Yamaguti algebra $\g$ reduces to a Lie tripe system $(\g,\Courant{\cdot,\cdot,\cdot})$,  then $(V;\mu)$  is a representation of the Lie triple systems $(\g,\Courant{\cdot,\cdot,\cdot})$; If $\mu=0$, $D_{\rho,\mu}=0$ and the Lie-Yamaguti algebra $\g$ reduces to a Lie algebra $(\g,[\cdot,\cdot])$, then $(V;\rho)$ is a representation  of the Lie algebra $(\g,[\cdot,\cdot])$.
\end{rmk}}
By a direct computation, we have the following proposition.

\begin{pro}
If $(V;\rho,\mu)$ is a representation of a Lie-Yamaguti algebra $(\g,[\cdot,\cdot],\Courant{\cdot,\cdot,\cdot})$. Then we have the following equalities
\begin{eqnarray*}
\label{RLYc}&&D([x,y],z)+D([y,z],x)+D([z,x],y)=0;\\
\label{RLY5a}&&D(\Courant{x,y,z},w)+D(z,\Courant{x,y,w})=[D(x,y),D(z,w)];\\
~ &&\mu(\Courant{x,y,z},w)=\mu(x,w)\mu(z,y)-\mu(y,w)\mu(z,x)-\mu(z,w)D(x,y),\label{RLY6}
\end{eqnarray*}
for all $x,y,z,w\in \g$.
\end{pro}

\begin{ex}\label{ad}
Let $(\g,[\cdot,\cdot],\Courant{\cdot,\cdot,\cdot})$ be a Lie-Yamaguti algebra. We define linear maps $\ad:\g \to \gl(\g)$ and $\frkR :\otimes^2\g \to \gl(\g)$ to be $x \mapsto \ad_x$ and $(x,y) \mapsto \mathfrak{R}_{x,y}$ respectively, where $\ad_xz=[x,z]$ and $\mathfrak{R}_{x,y}z=\Courant{z,x,y}$ for all $z \in \g$. Then $(\ad,\mathfrak{R})$ forms a representation of $\g$ on itself, where $\frkL:= D_{\ad,\frkR}$ is given by
$$\frkL_{x,y}z=\Courant{x,y,z}, \quad \forall z\in \g.$$
The representation $(\g;\ad,\frkR)$ is called the {\bf adjoint representation}.
\end{ex}
\emptycomment{
In the following, we give the notion of homomorphisms between representations of a Lie-Yamaguti algebra.
\begin{defi}
Let $(\g,[\cdot,\cdot],\Courant{\cdot,\cdot,\cdot})$ be a Lie-Yamaguti algebra and $(V;\rho_V,\mu_V)$ and $(W;\rho_W,\mu_W)$ be its two representations. A {\bf homomorphism} from $(V;\rho_V,\mu_V)$ and $(W;\rho_W,\mu_W)$ is a pair $(\phi_\g,\psi)$ consisting of a Lie-Yamaguti isomorphism $\phi_\g:\g \to \g$ and a linear map $\psi:V\to W$ such that for all $x,y\in \g,~v\in V,$
}

Representations of a Lie-Yamaguti algebra can be characterized by the semidirect product Lie-Yamaguti algebras.

\begin{pro}{\rm(\cite{Zhang1})}
Let $(\g,[\cdot,\cdot],\Courant{\cdot,\cdot,\cdot})$ be a Lie-Yamaguti algebra and $V$ a vector space. Suppose that $\rho:\g \to \gl(V)$ and $\mu:\otimes^2 \g \to \gl(V)$ are linear maps. Then $(V;\rho,\mu)$ is a representation of $(\g,[\cdot,\cdot],\Courant{\cdot,\cdot,\cdot})$ if and only if there is a Lie-Yamaguti algebra structure $([\cdot,\cdot]_{\ltimes},\Courant{\cdot,\cdot,\cdot}_{\ltimes})$ on the direct sum $\g \oplus V$ which is defined to be
\begin{eqnarray*}
\label{semi1}[x+u,y+v]_{\ltimes}&=&[x,y]+\rho(x)v-\rho(y)u,\\
\label{semi2}~\Courant{x+u,y+v,z+w}_{\ltimes}&=&\Courant{x,y,z}+D(x,y)w+\mu(y,z)u-\mu(x,z)v,
\end{eqnarray*}
for all $x,y,z \in \g, ~u,v,w \in V$. This Lie-Yamaguti algebra $(\g \oplus V,[\cdot,\cdot]_{\ltimes},\Courant{\cdot,\cdot,\cdot}_{\ltimes})$ is called the {\bf semidirect product Lie-Yamaguti algebra}, and is denoted by $\g \ltimes V$.
\end{pro}

Cohomology theory on Lie-Yamaguti algebras was founded in \cite{Yamaguti2}. Let $(V;\rho,\mu)$ be a representation of a Lie-Yamaguti algebra $(\g,[\cdot,\cdot],\Courant{\cdot,\cdot,\cdot})$, and denote the set of $p$-cochains by $C^p_{\rm LieY}(\g,V)~(p \geqslant 1)$, where
\begin{eqnarray*}
C^{n+1}_{\rm LieY}(\g,V)\triangleq
\begin{cases}
\Hom(\underbrace{\wedge^2\g\otimes \cdots \otimes \wedge^2\g}_n,V)\times \Hom(\underbrace{\wedge^2\g\otimes\cdots\otimes\wedge^2\g}_{n}\otimes\g,V), & \forall n\geqslant 1,\\
\Hom(\g,V), &n=0.
\end{cases}
\end{eqnarray*}

The coboundary map $\delta:C^p_{\rm LieY}(\g,V)\longrightarrow C^{p+1}_{\rm LieY}(\g,V)$ is given as follows:
\begin{itemize}
\item If $n\geqslant 1$, for any $(f,g)\in C^{n+1}_{\rm LieY}(\g,V)$, the coboundary map
$$\delta=(\delta_{\rm I},\delta_{\rm II}):C^{n+1}_{\rm LieY}(\g,V)\to C^{n+2}_{\rm LieY}(\g,V),$$
$$\qquad \qquad\qquad \qquad\qquad \quad (f,g)\mapsto\Big(\delta_{\rm I}(f,g),\delta_{\rm II}(f,g)\Big),$$
 is given by
\begin{eqnarray}
~\nonumber &&\Big(\delta_{\rm I}(f,g)\Big)(\frkX_1,\cdots,\frkX_{n+1})\\
~\label{cohomo1} &=&(-1)^n\Big(\rho(x_{n+1})g(\frkX_1,\cdots,\frkX_n,y_{n+1})-\rho(y_{n+1})g(\frkX_1,\cdots,\frkX_n,x_{n+1})\\
~\nonumber &&-g(\frkX_1,\cdots,\frkX_n,[x_{n+1},y_{n+1}])\Big)\\
~\nonumber &&+\sum_{k=1}^{n}(-1)^{k+1}D_{\rho,\mu}(\frkX_k)f(\frkX_1,\cdots,\hat{\frkX_k},\cdots,\frkX_{n+1})\\
~\nonumber &&+\sum_{1\leqslant k<l\leqslant n+1}(-1)^{k}f(\frkX_1,\cdots,\hat{\frkX_k},\cdots,\frkX_k\circ\frkX_l,\cdots,\frkX_{n+1}),\\
~ \nonumber&&\\
~\nonumber &&\Big(\delta_{\rm II}(f,g)\Big)(\frkX_1,\cdots,\frkX_{n+1},z)\\
~\label{cohomo2}&=&(-1)^n\Big(\mu(y_{n+1},z)g(\frkX_1,\cdots,\frkX_n,x_{n+1})-\mu(x_{n+1},z)g(\frkX_1,\cdots,\frkX_n,y_{n+1})\Big)\\
~\nonumber &&+\sum_{k=1}^{n+1}(-1)^{k+1}D_{\rho,\mu}(\frkX_k)g(\frkX_1,\cdots,\hat{\frkX_k},\cdots,\frkX_{n+1},z)\\
~\nonumber &&+\sum_{1\leqslant k<l\leqslant n+1}(-1)^kg(\frkX_1,\cdots,\hat{\frkX_k},\cdots,\frkX_k\circ\frkX_l,\cdots,\frkX_{n+1},z)\\
~\nonumber &&+\sum_{k=1}^{n+1}(-1)^kg(\frkX_1,\cdots,\hat{\frkX_k},\cdots,\frkX_{n+1},\Courant{x_k,y_k,z}),
\end{eqnarray}
where $\frkX_i=x_i\wedge y_i\in\wedge^2\g~(i=1,\cdots,n+1),~z\in \g$, and the notation $\frkX_k\circ\frkX_l$ means that
$$\frkX_k\circ\frkX_l:=\Courant{x_k,y_k,x_l}\wedge y_l+x_l\wedge\Courant{x_k,y_k,y_l}.$$
\item If $n=0$, for any element $f \in C^1_{\rm LieY}(\g,V)$, the coboundary map
$$\delta:C^1_{\rm LieY}(\g,V)\to C^2_{\rm LieY}(\g,V),$$
$$\qquad \qquad \qquad f\mapsto \Big(\delta_I(f),\delta_{II}(f)\Big),$$
is given by
\begin{eqnarray}
\label{1cochain}\Big(\delta_{\rm I}(f)\Big)(x,y)&=&\rho(x)f(y)-\rho(y)f(x)-f([x,y]),\\
~ \label{2cochain}\Big(\delta_{\rm II}(f)\Big)(x,y,z)&=&D_{\rho,\mu}(x,y)f(z)+\mu(y,z)f(x)-\mu(x,z)f(y)-f(\Courant{x,y,z}),\quad \forall x,y, z\in \g.
\end{eqnarray}
\end{itemize}

\begin{pro}{\rm (\cite{Yamaguti2})}
 With the notations above, for any $f\in C^1_{\rm LieY}(\g,V)$, we have
 \begin{eqnarray*}
 \delta_{\rm I}\Big(\delta_{\rm I}(f)),\delta_{\rm II}(f)\Big)=0\quad {\rm and} \quad\delta_{\rm II}\Big(\delta_{\rm I}(f)),\delta_{\rm II}(f)\Big)=0.
 \end{eqnarray*}
 Moreover, for all $(f,g)\in C^p_{\rm LieY}(\g,V)~(p\geqslant 2)$, we have
  \begin{eqnarray*}
  \delta_{\rm I}\Big(\delta_{\rm I}(f,g)),\delta_{\rm II}(f,g)\Big)=0\quad{\rm and} \quad \delta_{\rm II}\Big(\delta_{\rm I}(f,g)),\delta_{\rm II}(f,g)\Big) =0.
  \end{eqnarray*}
 In some literature, this conclusion is denote by $\delta\circ\delta=0.$
  \end{pro}
\emptycomment{
\begin{defi}
With the above notations, let $(f,g)$ in $C^p_{\rm LieY}(\g,V))$ (resp. $f\in C^1_{\rm LieY}(\g,V)$ for $p=1$) be a $p$-cochain. If it satisfies $\delta(f,g)=0$ (resp. $\delta(f)=0$), then it is called a $p$-cocycle. If there exists $(h,s)\in C^{p-1}_{\rm LieY}(\g,V)$,~(resp. $t\in C^1(\g,V)$, if $p=2$) such that $(f,g)=\delta(h,s)$~(resp. $(f,g)=\delta(t)$), then it is called a $p$-coboundary ($p\geqslant 2$). The set of $p$-cocycles and that of $p$-coboundaries is denoted by $Z^p_{\rm LieY}(\g,V)$ and $B^p_{\rm LieY}(\g,V)$ respectively. The resulting $p$-cohomology group is defined to be the factor space
$$H^p_{\rm LieY}(\g,V)=Z^p_{\rm LieY}(\g,V)/B^p_{\rm LieY}(\g,V)~(p\geqslant2).$$ In particular, we have
$$H^1_{\rm LieY}(\g,V)=\{f\in C^1_{\rm LieY}(\g,V):\delta (f)=0\}.$$
\end{defi}}

\section{Cohomology and Linear deformations of \LYP pairs}
In this section, we establish the cohomology theory of \LYP pairs, and use this type of cohomology to classify linear deformations of \LYP pairs.

\medskip
\subsection{Cohomology of \LYP pairs}

Recall that in this paper, we use a terminology \LYP pair to denote a pair $\Big((\g,\br,,\ltp),(V;\rho,\mu)\Big)$, where $(\g,\br,,\ltp)$ is a \LYA, and $(V;\rho,\mu)$ is a representation of $\g$. We also denote a \LYP pair by $(\g,\rho,\mu)$ without ambiguity.

\medskip
For the moment, let $\g$ and $V$ be vector spaces, where elements in $\g$ are denoted by $x,y,z,x_i,y_i$ and those in $V$ by $u,v,w,u_i,v_i$. Denote by $\huaC^\bullet(\g,\rho,\mu)=\oplus_{n=0}^\infty \huaC^n(\g,\rho,\mu)$, where $\huaC^0(\g,\rho,\mu)$ is defined to be $0$, $\huaC^1(\g,\rho,\mu)$ is defined to be $\gl(\g)\times\gl(V)$, and for $n\geqslant1$,
{\footnotesize
\begin{eqnarray*}
~ &&\huaC^{n+1}(\g,\rho,\mu)\\
~&:=&\Big(\Hom(\underbrace{\wedge^2\g\otimes\cdots\otimes \wedge^2\g}_n,\g)\times\Hom(\underbrace{\wedge^2\g\otimes\cdots\otimes \wedge^2\g}_n\otimes\g,\g)\Big)\\
~ &&\oplus\Big(\Hom(\underbrace{\wedge^2\g\otimes\cdots\otimes (\g\wedge V)}_n,V)\times\Hom(\underbrace{\wedge^2\g\otimes\cdots\otimes \wedge^2\g}_n\otimes V,V)\Big)\\
~ &&\oplus\Big(\oplus_{i=1}^{n-1}\Hom(\underbrace{\wedge^2\g\otimes\cdots\otimes(\g\wedge V)\otimes\cdots\wedge^2\g }_n,V)\times\Big(\oplus_{j=1}^n\Hom(\underbrace{\wedge^2\g\otimes\cdots\otimes(\g\wedge V)\otimes\cdots \wedge^2\g}_n\otimes \g,V)\Big)\Big),
\end{eqnarray*}}
where $\g\wedge V$ occurring in the third term is at the $i$-th position.

Assume that $n\geqslant1$ now, and $(f,g)\in \huaC^{n+1}(\g,\rho,\mu)$.
\begin{itemize}
\item[\rm(a)] For a pair of given maps
$$(f,g)\in \Hom(\underbrace{\wedge^2\g\otimes\cdots\otimes \wedge^2\g}_n,\g)\times\Hom(\underbrace{\wedge^2\g\otimes\cdots\otimes \wedge^2\g}_n\otimes\g,\g),$$
 we define a pair of linear maps $(\hat f,\hat g)\in C_{\rm LieY}^{n+1}(\g\oplus V,\g\oplus V)$ to be
\begin{eqnarray*}
\hat f\Big((x_1,u_1)\wedge(y_1,v_1),\cdots,(x_n,u_n)\wedge(y_n,v_n)\Big)&=&\Big(f(x_1\wedge y_1,\cdots,x_n\wedge y_n),0\Big),\\
\hat g\Big((x_1,u_1)\wedge(y_1,v_1),\cdots,(x_n,u_n)\wedge(y_n,v_n),(x_{n+1},u_{n+1})\Big)&=&\Big(g(x_1\wedge y_1,\cdots,x_n\wedge y_n,x_{n+1}),0\Big).
\end{eqnarray*}
\item[(b)] For a pair of given maps
 $$(f,g)\in \Hom(\underbrace{\wedge^2\g\otimes\cdots\otimes (\g\wedge V)}_n,V)\times\Hom(\underbrace{\wedge^2\g\otimes\cdots\otimes \wedge^2\g}_n\otimes V,V),$$
we define a pair of linear maps $(\hat f,\hat g)\in C_{\rm LieY}^{n+1}(\g\oplus V,\g\oplus V)$ to be
\begin{eqnarray*}
~ &&\hat f\Big((x_1,u_1)\wedge(y_1,v_1),\cdots,(x_n,u_n)\wedge(y_n,v_n)\Big)\\
&=&\Big(0,f(x_1\wedge y_1,\cdots,x_{n-1}\wedge y_{n-1},(x_n,v_n))-f(x_1\wedge y_1,\cdots,x_{n-1}\wedge y_{n-1},(y_n,u_n))\Big),\\
~ &&\hat g\Big((x_1,u_1)\wedge(y_1,v_1),\cdots,(x_n,u_n)\wedge(y_n,v_n),(z,w)\Big)\\
~&=&\Big(0,g(x_1\wedge y_1,\cdots,x_n\wedge y_n,w)-g(x_1\wedge y_1,\cdots,z\wedge y_n,u_n)+g(x_1\wedge y_1,\cdots,z\wedge x_n,v_n)\Big).\\
\end{eqnarray*}
\item[(c)] For a pair of given maps
 $$(f,g)\in \Hom(\underbrace{\wedge^2\g\otimes\cdots\otimes (\g\wedge V)\otimes\cdots\otimes\wedge^2\g}_n,V)\times\Hom(\underbrace{\wedge^2\g\otimes\cdots\otimes(\g\wedge V)\otimes\cdots \wedge^2\g}_n\otimes \g,V),$$
we define a pair of linear maps $(\hat f,\hat g)\in C_{\rm LieY}^{n+1}(\g\oplus V,\g\oplus V)$ to be
\begin{eqnarray*}
~ &&\hat f\Big((x_1,u_1)\wedge(y_1,v_1),\cdots,(x_n,u_n)\wedge(y_n,v_n)\Big)\\
&=&\Big(0,f(x_1\wedge y_1,\cdots,(x_i,v_i)\cdots,(x_n,u_n)\wedge(y_n,v_n))-f(x_1\wedge y_1,\cdots,(y_i,u_i),\cdots,(x_n,u_n)\wedge(y_n,v_n))\Big),\\
~ &&\hat g\Big((x_1,u_1)\wedge(y_1,v_1),\cdots,(x_n,u_n)\wedge(y_n,v_n),(z,w)\Big)\\
~&=&\Big(0,g(x_1\wedge y_1,\cdots, (x_i,v_i),\cdots,x_n\wedge y_n,z)-g(x_1\wedge y_1,\cdots, (y_i,u_i),\cdots,x_n\wedge y_n,z)\Big).
\end{eqnarray*}
\end{itemize}
Moreover, for a given linear map $f=(f_1,f_2)\in\huaC^1(\g,\rho,\mu)=\Hom(\g,\g)\times\Hom(V,V)$, we define a linear map $\hat f\in C_{\rm LieY}^1(\g\oplus V,\g\oplus V)$ to be
\begin{eqnarray*}
\hat f(x,u)=(f_1(x),f_2(u)),\quad \forall (x,u)\in \g\oplus V.
\end{eqnarray*}

A pair of linear maps $(\hat f,\hat g)$ (resp. $\hat f$) defined above is called the {\bf lift} of $(f,g)$ (resp. $f$). It is obvious that $\huaC^n(\g,\rho,\mu)$ is a subspace of $C^n_{\rm LieY}(\g\oplus V,\g\oplus V)$ via the lift map.

Let $(\g,\rho,\mu)$ be a \LYP pair. For $(f,g)=\Big((f_1,g_1),(f_2,g_2),(f_3,g_3)\Big)\in \huaC^{n+1}(\g,\rho,\mu)~(n\geqslant 1)$, write the coboundary map
$$\Delta:\huaC^{n+1}(\g,\rho,\mu)\to \huaC^{n+2}(\g,\rho,\mu)$$
as
$$\Delta(f,g)=\Big(\big(\Delta(f,g)\big)_1,\big(\Delta(f,g)\big)_2,\big(\Delta(f,g)\big)_3\Big),$$
where
\begin{itemize}
\item $\Big(\Delta(f,g)\Big)_1=\delta_{\mathsf{LieY}}(f_1,g_1)$ is the coboundary map given in \eqref{cohomo1} and \eqref{cohomo2} with coefficients in the adjoint representation.
\item $\Big(\Delta(f,g)\Big)_2$ is defined to be
\begin{eqnarray*}
~ &&\Big(\Delta_{\rm I}(f,g)\Big)_2(\frkX_1,\cdots,\frkX_n,(x,v))\\
~ &=&(-1)^{n}\Big(\rho(x)g_2(\frkX_1,\cdots,\frkX_n,v)+\rho(g_1(\frkX_1,\cdots,\frkX_n,x))v-g_2(\frkX_1,\cdots,\frkX_n,\rho(x)v)\Big)\\
~ &&+\sum_{k=1}^n(-1)^{k+1}D(x_k,y_k)f_2(\frkX_1,\cdots,\hat\frkX_k,\cdots,\frkX_n,(x,v))\\
~ &&+\sum_{1\leqslant k<l\leqslant n}(-1)^kf_2(\frkX_1,\cdots,\hat\frkX_k,\cdots,\frkX_k\circ \frkX_l,\cdots,\frkX_n,(x,v))\\
~ &&+\sum_{1\leqslant k<n+1}(-1)^kf_2(\frkX_1,\cdots,\hat\frkX_k,\cdots,\frkX_n,(\Courant{x_k,y_k,x},v)+(x,D(x_k,y_k)v)),\\
~ &&\Big(\Delta_{\rm II}(f,g)\Big)_2(\frkX_1,\cdots,\frkX_{n+1},v)\\
~ &=&(-1)^n\Big(D(g_1(\frkX_1,\cdots,\frkX_n,x_{n+1}),y_{n+1})v-D(g_1(\frkX_1,\cdots,\frkX_n,y_{n+1}),x_{n+1})v\Big)\\
~ &&+\sum_{k=1}^{n+1}(-1)^{k+1}D(x_k,y_k)g_2(\frkX_1,\cdots,\hat\frkX_k,\cdots,\frkX_{n+1},v)\\
~ &&+\sum_{1\leqslant k<l\leqslant n+1}(-1)^kg_2(\frkX_1,\cdots,\hat\frkX_k,\cdots,\frkX_k\circ \frkX_l,\cdots,\frkX_{n+1},v)\\
~ &&+\sum_{k=1}^{n+1}(-1)^kg_2(\frkX_1,\cdots,\hat\frkX_k,\cdots,\frkX_{n+1},D(x_k,y_k)v).
\end{eqnarray*}
\item $\Big(\Delta(f,g)\Big)_3$ is defined to be
\begin{eqnarray*}
~ &&\Big(\Delta_{\rm I}(f,g)\Big)_3(\frkX_1,\cdots,(x_i,v),\cdots,\frkX_{n+1})\\
~ &=&(-1)^{n}\Big(\rho(x_{n+1})g_3(\frkX_1,\cdots,(x_i,v),\cdots,\frkX_n,y_{n+1})-\rho(y_{n+1})g_3(\frkX_1,\cdots,(x_i,v),\cdots,\frkX_n,x_{n+1})\\
~ &&-g_3(\frkX_1,\cdots,(x_i,v),\cdots,[x_{n+1},y_{n+1}]\Big)\\
~ &&+\sum_{k=1,\atop k\neq i}^n(-1)^{k+1}D(x_k,y_k)f_3(\frkX_1,\cdots,\hat\frkX_k,\cdots,(x_i,v),\cdots,\frkX_{n+1})\\
~ &&+\sum_{k<l\leqslant n+1,\atop k,l\neq i}(-1)^kf_3(\frkX_1,\cdots,(x_i,v),\cdots,\hat\frkX_k,\cdots,\frkX_k\circ \frkX_l,\cdots,\frkX_{n+1})\\
~ &&+(-1)^i\mu(x_i,f_1(\frkX_1,\cdots,\hat{(x_i,v)},\cdots,\frkX_{n+1}))v\\
~ &&+\sum_{k> i}(-1)^if_3(\frkX_1,\cdots,\hat{(x_i,v)},\cdots,(x_k,\mu(x_i,y_k)v)-(y_k,\mu(x_i,x_k)v),\cdots,\frkX_{n+1})\\
~ &&+\sum_{k<i}(-1)^kf_3(\frkX_1,\cdots,\hat\frkX_k,\cdots,\frkX_{i-1},(\Courant{x_k,y_k,x_i},v)+(x_i,D(x_k,y_k)v),\frkX_{i+1},\cdots,\frkX_{n+1}),\\
~ &&\Big(\Delta_{\rm II}(f,g)\Big)_3(\frkX_1,\cdots,(x_i,v),\cdots,\frkX_{n+1},x)\\
~ &=&(-1)^{n}\Big(\mu(y_{n+1},x)g_3(\frkX_1,\cdots,(x_i,v),\cdots,\frkX_{n+1},x_{n+1})
-\mu(x_{n+1},x)g_3(\frkX_1,\cdots,(x_i,v),\cdots,\frkX_{n+1},y_{n+1})\Big)\\
~ &&+\sum_{k=1,\atop k\neq i}^{n+1}(-1)^{k+1}D(x_k,y_k)g_3(\frkX_1,\cdots,(x_i,v),\cdots,\hat\frkX_k,\cdots,\frkX_{n+1},x)\\
~ &&+\sum_{k<l\leqslant n+1,\atop k,l\neq i}(-1)^kg_3(\frkX_1,\cdots,(x_i,v),\cdots,\hat\frkX_k,\cdots,\frkX_k\circ \frkX_l,\cdots,\frkX_{n+1},x)\\
~ &&+\sum_{k=1}^{n+1}g_3(\frkX_1,\cdots,(x_i,v),\cdots,\hat\frkX_k,\cdots,\frkX_{n+1},\Courant{x_k,y_k,x})\\
~ &&+(-1)^i\Big(\mu(x_i,g_1(\frkX_1,\cdots,\hat{(x_i,v)},\cdots,\frkX_{n+1},x))v-g_2(\frkX_1,\cdots,\hat{(x_i,v)},\cdots,\frkX_{n+1},\mu(x_i,x)v)\\
~ &&\quad+g_3(\frkX_1,\cdots,\hat{(x_i,v)},\cdots,(y_l,\mu(x_i,x_l)v)-(x_l,\mu(x_i,y_l)v),\cdots,\frkX_{n+1},x)\Big).
\end{eqnarray*}
\end{itemize}

For a pair of given maps $f=\Big(f_1,f_2\Big)\in \huaC^1(\g,\rho,\mu)=\Hom(\g,\g)\times\Hom(V,V)$, the coboundary map
$$\Delta:\huaC^1(\g,\rho,\mu)\to \huaC^2(\g,\rho,\mu)$$
is written as $\Big(\Delta(f)_1,\Delta(f)_2,\Delta(f)_3\Big)$, where
\begin{itemize}
\item $\Big(\Delta (f)\Big)_1=\delta_{\mathsf{LieY}} f_1$ is the coboundary map given in \eqref{1cochain} and \eqref{2cochain} with coefficients in the adjoint representation.
\item $\Big(\Delta (f)\Big)_2$ is given as follows
\begin{eqnarray*}
\Big(\Delta_{\rm I} (f)\Big)_2(x,v)&=&\rho(x)f_2(v)+\rho(f_1(x))v-f_2(\rho(x)v),\\
\Big(\Delta_{\rm II} (f)\Big)_2(x,y,v)&=&D(x,y)f_2(v)+D(f_1(x),y)v+D(x,f_1(y))v-f_2(D(x,y)v).
\end{eqnarray*}
\item $\Big(\Delta (f)\Big)_3$ is given as follows
\begin{eqnarray*}
\Big(\Delta_{\rm I} (f)\Big)_3(x,v)&=&\rho(x)f_2(v)+\rho(f_1(x))v-f_2(\rho(x)v),\\
\Big(\Delta_{\rm II} (f)\Big)_3(v,x,y)&=&\mu(x,y)f_2(v)+\mu(f_1(x),y)v+\mu(x,f_1(y))v-f_2(\mu(x,y)v).
\end{eqnarray*}
\end{itemize}

\begin{thm}
With the above notations, we have $\Delta\circ \Delta=0$. Thus we obtain a well-defined complex $(\huaC^\bullet(\g,\rho,\mu),\Delta)$.
\end{thm}
\noindent{\emph{Sketch of Proof.}
 Consider the semidirect product \LYA ~$(\g\ltimes V,\br__{\ltimes},\ltp_{\ltimes})$, and take its corresponding cohomology $(C^\bullet_{\mathsf{LieY}}(\g\oplus V,\g\oplus V),\delta)$ with coefficients in the adjoint representation. It is easy to see that $(\huaC^\bullet(\g,\rho,\mu),\Delta)$ is a subcomplex of $(C^\bullet_{\mathsf{LieY}}(\g\oplus V,\g\oplus V),\delta)$. Thus for any $(f,g)\in\huaC^n(\g,\rho,\mu)$, we have that $(\Delta\circ \Delta)(f,g)_1=(\Delta\circ \Delta)(f,g)_2=(\Delta\circ \Delta)(f,g)_3=0$. The conclusion thus follows.}
\qed

\begin{defi}
Let $(\g,\rho,\mu)$ be a \LYP pair. We call the cochain complex $(\huaC^\bullet(\g,\rho,\mu),\Delta)$ the {\bf cohomology of the \LYP pair $(\g,\rho,\mu)$}. The corresponding $n$-th cohomology group is denoted by
$$\huaH^n(\g,\rho,\mu)=\huaZ^n(\g,\rho,\mu)/\huaB^n(\g,\rho,\mu),$$
where $\huaZ^n(\g,\rho,\mu)$ and $\huaB^n(\g,\rho,\mu)$ are sets of $n$-cocycles and $n$-coboundaries respectively.
\end{defi}

\subsection{Linear deformations of \LYP pairs}
In this subsection, we characterize linear deformations of \LYP pairs by the cohomology theory established in the previous subsection. The following definition is standard.

\begin{defi}\label{defi:homo}
Let $\Big((\g,\br_{_\g},\ltp_\g),(V;\rho_V,\mu_V)\Big)$ and $\Big((\h,\br_{_\h},\ltp_\h),(W;\rho_W,\mu_W)\Big)$ be two \LYP pairs.
A {\bf homomorphism} from $(\g,\rho_V,\mu_V)$ to $(\h,\rho_W,\mu_W)$ is a pair $(\phi,\varphi)$, where $\phi:\g \to \h$ is a Lie-Yamaguti algebra homomorphism and $\varphi:V\to W$ is a homomorphism of representations from $(V;\rho,\mu)$ to $(W;\varrho,\varpi)$, i.e.,
\begin{eqnarray}
\label{homo1}\varphi\Big(\rho_V(x)v\Big)&=&\rho_W\Big(\phi(x)\Big)\big(\varphi(v)\big),\\
\label{homo2}\varphi\Big(\mu_V(x,y)v\Big)&=&\mu_W\Big(\phi(x),\phi(y)\Big)\big(\varphi(v)\big), \quad \forall x,y\in \g,~v\in V.
\end{eqnarray}
\end{defi}
By Eqs. \eqref{rep}, \eqref{homo1}, and \eqref{homo2}, we deduce that
\begin{eqnarray}
\label{homo3}\varphi\Big(D_V(x,y)v\Big)&=&D_W\Big(\phi(x),\phi(y)\Big)\big(\varphi(v)\big),\quad \forall x,y\in \g,~v\in V,
\end{eqnarray}
where $D_V=D_{\rho_{V},\mu_V}$ and $D_W=D_{\rho_W,\mu_W}$.

Let $\Big((\g,[\cdot,\cdot],\Courant{\cdot,\cdot,\cdot}),(V;\rho,\mu)\Big)$ be a \LYP pair. Let $\phi:\wedge^2\g\to \g,~\varphi_i:\wedge^2\g\otimes\g \to \g$ and $\varrho:\g\to \gl(V),~\varpi_i:\otimes^2\g \to \gl(V)~(i=1,2)$ be linear maps. Consider a $t$-parameterized family of bracket operations and linear maps:
\begin{eqnarray*}
[x,y]_t&=&[x,y]+t\phi(x,y),\\
~\Courant{x,y,z}_t&=&\Courant{x,y,z}+t\varphi_1(x,y,z)+t^2\varphi_2(x,y,z),\\
\rho_t(x)&=&\rho(x)+t\varrho(x),\\
\mu_t(x,y)&=&\mu(x,y)+t\varpi_1(x,y)+t^2\varpi_2(x,y)
\end{eqnarray*}
By \eqref{rep}, we have
$$D_t:= D_{\rho_t,\mu_t}=D+t\huaD_1+t^2\huaD_2$$
where
\begin{eqnarray*}
\huaD_1(x,y)&=&\varpi_1(y,x)-\varpi_1(x,y)+[\rho(x),\varrho(y)]+[\varrho(x),\rho(y)]-\varrho([x,y]),\\
~\huaD_2(x,y)&=&\varpi_2(y,x)-\varpi_2(x,y)+[\varrho(x),\varrho(y)].
\end{eqnarray*}
It is easy to see that both $\huaD_1$ and $\huaD_2$ are skew-symmetric, and so is $D_t$.

\begin{defi}
Let $\Big((\g,\br,,\ltp),(V;\rho,\mu)\Big)$ be a \LYP pair. With the above notations, if $\Big((\g,[\cdot,\cdot]_t,\Courant{\cdot,\cdot,\cdot}_t),(V,\rho_t,\mu_t)\Big)$ are a family of \LYP pairs for all $t$, we say that $(\phi,\varphi_i;\varrho,\varpi_i)~(i=1,2)$ generates a {\bf linear deformation} of the \LYP pair $\Big((\g,[\cdot,\cdot],\Courant{\cdot,\cdot,\cdot}),(V;\rho,\mu)\Big)$.
\end{defi}

\begin{pro}\label{fund:pro}
Let $\Big((\g,[\cdot,\cdot],\Courant{\cdot,\cdot,\cdot}),(V;\rho,\mu)\Big)$ be a \LYP pair. If a collection of linear maps $(\phi,\varphi_i;\varrho,\varpi_i)~(i=1,2)$ generates a linear deformation of the \LYP pair $(\g,\rho,\mu)$, then we have
\begin{itemize}
\item [\rm (i)] $\Big((\phi,\varphi_1),(\varrho,\huaD_1),(0,0)\Big)\in \huaC^2(\g,\rho,\mu)$ is a $2$-cocycle of $\Delta$, i.e., $$\Delta\Big((\phi,\varphi_1),(\varrho,\huaD_1),(0,0)\Big)=0;$$
\item [\rm (ii)] $\Big((\g,\phi,\varphi_2),(V;\varrho,\varpi_2)\Big)$ is a \LYP pair.
\end{itemize}
\end{pro}
\begin{proof}
The computation is straightforward, so we omit the details.
\end{proof}
\emptycomment{
\begin{rmk}
Note that if $\Delta\Big((\phi,\varphi_1),(\varrho,\huaD_1),(0,0)\Big)=0$, i.e., $\Big((\phi,\varphi_1),(\varrho,\huaD_1),(0,0)\Big)\in \huaZ^2(\g,\rho,\mu)$, then we have
\begin{eqnarray*}
\Big(\Delta_{\rm I}(\varrho,\huaD_1)\Big)_2(x,y,z,v)=0\quad {\rm and} \quad\Big(\Delta_{\rm II}(\varrho,\huaD_1)\Big)_2(x,y,z,v)=0.
\end{eqnarray*}
It is easy to see that $\Big(\Delta_{\rm I}(\varrho,\huaD_1)\Big)_2(x,y,z,v)=0$ is equivalent to $$\rho(\varphi_1(x,y,z))+\varrho(\Courant{x,y,z})=[D(x,y),\varrho(z)]+[\huaD_1(x,y),\rho(z)],$$
while $\Big(\Delta_{\rm II}(\varrho,\huaD_1)\Big)_2(x,y,z,v)=0$ is equivalent to
{\footnotesize
$$D(\varphi_1(x,y,z),w)+\huaD_1(\Courant{x,y,z},w)+D(z,\varphi_1(x,y,w))+\huaD_1(z,\Courant{x,y,w})=[D(x,y),\huaD_1(z,w)]+[\huaD_1(x,y),D(z,w)],$$
}
which is exactly the coefficients of $t$ in $D_t(\Courant{x,y,z},w)+D_t(z,\Courant{x,y,w})=[D_t(x,y),D_t(z,w)].$
\end{rmk}}

Let $\Big((\g,[\cdot,\cdot],\Courant{\cdot,\cdot,\cdot}),(V;\rho,\mu)\Big)$ be a \LYP pair, and suppose that $(\g,[\cdot,\cdot]_t,\Courant{\cdot,\cdot,\cdot}_t;\rho_t,\mu_t)$ and $(\g,[\cdot,\cdot]_t',\Courant{\cdot,\cdot,\cdot}_t';\rho_t',\mu_t')$ are two linear deformations of $\Big((\g,[\cdot,\cdot],\Courant{\cdot,\cdot,\cdot}),(V;\rho,\mu)\Big)$. In the following, we denote $D_{\rho_t,\mu_t}$ and $D_{\rho_t',\mu_t'}$ by $D_t$ and $D_t'$ respectively.

\begin{defi}\label{defi:deformation}
\begin{itemize}
\item [\rm (i)]Two linear deformations $(\g,[\cdot,\cdot]_t,\Courant{\cdot,\cdot,\cdot}_t;\rho_t,\mu_t)$ and $(\g,[\cdot,\cdot]_t',\Courant{\cdot,\cdot,\cdot}_t';\rho_t',\mu_t')$ of $(\g,\rho,\mu)$ are called {\bf equivalent} if there exists an isomorphism $({\Id}_\g+tN,{\Id}_V+tS)$ from $(\g,[\cdot,\cdot]_t',\Courant{\cdot,\cdot,\cdot}_t';\rho_t',\mu_t')$ to $(\g,[\cdot,\cdot]_t,\Courant{\cdot,\cdot,\cdot}_t;\rho_t,\mu_t)$.
\item [\rm (ii)] A linear deformation $(\g,[\cdot,\cdot]_t,\Courant{\cdot,\cdot,\cdot}_t;\rho_t,\mu_t)$ is said to be {\bf trivial} if it is equivalent to $(\g,[\cdot,\cdot],\Courant{\cdot,\cdot,\cdot};\rho,\mu)$ itself.
\end{itemize}
\end{defi}

Suppose that two deformations $(\g,\br__t,\ltp_t;\rho_t,\mu_t)$ and $(\g,\br__t',\ltp_t';\rho_t',\mu_t')$ are equivalent, for $(N,S)\in \huaC^1(\g,\rho,\mu)=\gl(\g)\times\gl(V)$, by Definition \ref{defi:homo} and Definition \ref{defi:deformation}, we obtain that
\begin{eqnarray}
    \label{ee1}({\Id}_\g+tN)[x,y]_t'&=&[({\Id}_\g+tN)x,({\Id}_\g+tN)y]_t,\\
    \label{ee2}({\Id}_\g+tN)\Courant{x,y,z}_t'&=&\Courant{({\Id}_\g+tN)x,({\Id}_\g+tN)y,({\Id}_\g+tN)z}_t,\\
    \label{eqv1}({\Id}_V+tS)\circ \rho_t'(x)&=&\rho_t\Big(({\Id}_\g+tN)x\Big)\circ ({\Id}_V+tS),\\
    \label{eqv2}({\Id}_V+tS)\circ \mu_t'(x,y)&=&\mu_t\Big(({\Id}_\g+tN)x,({\Id}_\g+tN)y\Big)\circ ({\Id}_V+tS),\\
    \label{eqv3}({\Id}_V+tS)\circ D_t'(x,y)&=&D_t\Big(({\Id}_\g+tN)x,({\Id}_\g+tN)y\Big)\circ ({\Id}_V+tS),
    \end{eqnarray}
which implies that
\begin{eqnarray*}
\phi'(x,y)-\phi(x,y)&=&[Nx,y]+[x,Ny]-N[x,y]=\Big(\Delta_{\rm I}(N,S)\Big)_1(x,y),\\
\varphi_1'(x,y,z)-\varphi_1(x,y,z)&=&\Courant{Nx,y,z}+\Courant{x,Ny,z}+\Courant{x,y,Nz}-N\Courant{x,y,z}=\Big(\Delta_{\rm II}(N,S)\Big)_1(x,y,z),\\
\varrho'(x)v-\varrho(x)v&=&\rho(x)(Sv)+\rho(Nx)v-S(\rho(x)v)=\Big(\Delta_{\rm I}(N,S)\Big)_2(x,v),\\
\varpi_1'(x,y)v-\varpi_1(x,y)v&=&\mu(x,Ny)v+\mu(Nx,y)v+\mu(x,y)(Sv)-S(\mu(x,y)v)=\Big(\Delta_{\rm II}(N,S)\Big)_2(x,y,v),\\
\huaD_1'(x,y)v-\huaD_1(x,y)v&=&D(x,Ny)v+D(Nx,y)v+D(x,y)(Sv)-S(D(x,y)v)=\Big(\Delta_{\rm II}(N,S)\Big)_3(v,x,y).
\end{eqnarray*}

Thus, we have the following theorem.
\begin{thm}
If two linear deformations $(\g,\br__t,\ltp_t;\rho_t,\mu_t)$ and $(\g,\br__t',\ltp_t';\rho_t',\mu_t')$ of a \LYP pair $\Big((\g,\br,,\ltp),(V;\rho,\mu)\Big)$ are equivalent, which are generated by $(\phi,\varphi_i;\varrho,\varpi_i)$ and $(\phi',\varphi_i';\varrho',\varpi_i')$ respectively. Then $(\phi+\varrho,\varphi_1+\varpi_1)$ and $(\phi'+\varrho',\varphi_1'+\varpi_1')$ are in the same cohomology class of $\huaH^2(\g,\rho,\mu)$.
\end{thm}

By a straightforward computation,  a linear deformation $(\g,[\cdot,\cdot]_t,\Courant{\cdot,\cdot,\cdot}_t;\rho_t,\mu_t)$ of a \LYP pair $\Big((\g,[\cdot,\cdot],\Courant{\cdot,\cdot,\cdot}),(V;\rho,\mu)\Big)$ is trivial if and only if the following equalities hold:
\begin{eqnarray}
\begin{cases}\label{Nij1}
\phi(x,y)&=[Nx,y]+[x,Ny]-N[x,y],\\
~N\phi(x,y)&=[Nx,Ny],\\
\varphi_1(x,y,z)&=\Courant{Nx,y,z}+\Courant{x,Ny,z}+\Courant{x,y,Nz}-N\Courant{x,y,z},\\
\varphi_2(x,y,z)&=\Courant{Nx,Ny,z}+\Courant{x,Ny,Nz}+\Courant{Nx,y,Nz}-N\varphi_1(x,y,z),\\
N\varphi_2(x,y,z)&=\Courant{Nx,Ny,Nz},\\
\end{cases}\\
\begin{cases}\label{Nij4}
\varrho(x)&=\rho(Nx)+\rho(x)\circ S-S\circ\rho(x),\\
\rho(Nx)\circ S&=S\circ \varrho(x),\\
\varpi_1(x,y)&=\mu(Nx,y)+\mu(x,Ny)+\mu(x,y)\circ S-S\circ \mu(x,y),\\
\varpi_2(x,y)&=\mu(Nx,y)\circ S+\mu(x,Ny)\circ S+\mu(Nx,Ny)-S\circ \varpi_1(x,y),\\
~\mu(Nx,Ny)\circ S&=S\circ \varpi_2(x,y).
\end{cases}
\end{eqnarray}

\begin{defi}
Let $\Big((\g,[\cdot,\cdot],\Courant{\cdot,\cdot,\cdot}),(V;\rho,\mu)\Big)$ be a \LYP pair, $N\in \gl(\g)$ and $S\in \gl(V)$ linear maps on $\g$ and $V$ respectively. If Eqs. \eqref{Nij1}-\eqref{Nij4} hold, then $(N,S)$ is called a {\bf Nijenhuis structure} on $(\g,\rho,\mu)$.
\end{defi}

\begin{rmk}
Note that \eqref{Nij1} means that $N$ is a Nijenhuis operator on the Lie-Yamaguti algebra $(\g,[\cdot,\cdot],\Courant{\cdot,\cdot,\cdot})$, where $\phi:=[\cdot,\cdot]_N$ and $\varphi_2:=\ltp_N$ are deformed brackets. Besides, $(\br_{_N},\ltp_N)$  is a \LYA ~structure on $\g$ and thus $N$ is a homomorphism from \LYA ~$(\g,\br_{_N},\ltp_N)$ to $(\g,\br,,\ltp)$. See \cite{ShengZhao} for more details of Nijenhuis operators on \LYA s.
\end{rmk}

Note that Eqs. \eqref{Nij4} is equivalent to the following equalities:
\begin{eqnarray}
\label{Nijstru1}\rho(Nx)(Sv)&=&S\Big(\rho(x)(Sv)-\rho(Nx)v\Big)-S^2(\rho(x)v),\\
\label{Nijstru2}\mu(Nx,Ny)(Sv)&=&S\Big(\mu(Nx,y)(Sv)+\mu(x,Ny)(Sv)+\mu(Nx,Ny)v\Big)\\
~ \nonumber&&-S^2\Big(\mu(Nx,y)v+\mu(x,Ny)v+\mu(x,y)(Sv)\Big)+S^3\Big(\mu(x,y)v\Big).
\end{eqnarray}

\begin{lem}\label{fundlem}
Let $(N,S)$ be a Nijenhuis structure on a \LYP pair $\Big((\g,[\cdot,\cdot],\Courant{\cdot,\cdot,\cdot}),(V;\rho,\mu)\Big)$. Then we have
\begin{eqnarray}
\label{Nijstru3}D(Nx,Ny)(Sv)&=&S\Big(D(Nx,y)(Sv)+D(x,Ny)(Sv)+D(Nx,Ny)v\Big)\\
~ \nonumber&&-S^2\Big(D(Nx,y)v+D(x,Ny)v+D(x,y)(Sv)\Big)+S^3\Big(D(x,y)v\Big).
\end{eqnarray}
\end{lem}
\begin{proof}
The proof only involves computation, so we omit the details.
\end{proof}
\begin{ex}
Let $N$ be a Nijenhuis operator on a \LYA ~$(\g,\br,,\ltp)$. Then the pair $(N,N)$ is a Nijenhuis structure on the \LYP pair $\Big((\g,\br,,\ltp),(\g;\ad,\frkR)\Big)$.
\end{ex}

Obviously, a trivial deformation of a \LYP pair gives rise to a Nijenhuis structure. Furthermore, the converse is also valid. First, we need a lemma.
\begin{lem}\label{prelem}
Let $\Big((\g,[\cdot,\cdot],\Courant{\cdot,\cdot,\cdot}),(V;\rho,\mu)\Big)$ be a \LYP pair, and $\h$ and $W$ vector spaces, where $\h$ is endowed with a pair of two operations $([\cdot,\cdot]',\Courant{\cdot,\cdot,\cdot}')$ and $W$ is endowed with two linear maps $\varrho:W\to \gl(\h),~\varpi:\otimes^2W\to \gl(\h)$. If there exist two isomorphisms between vector spaces $f:\h\to\g$ and $g:W\to V$, such that
\begin{eqnarray*}
f([x,y]')&=&[f(x),f(y)],\\
~f(\Courant{x,y,z}')&=&\Courant{f(x),f(y),f(z)},\\
~g(\varrho(x)w)&=&\rho(f(x))(g(w)),\\
~g(\varpi(x,y)w)&=&\mu(f(x),f(y))(g(w)),\quad \forall x,y,z\in \h,w\in W.
\end{eqnarray*}
then $\Big((\h,[\cdot,\cdot]',\Courant{\cdot,\cdot,\cdot}'),(W;\varrho,\varpi)\Big)$ is a \LYP pair
\end{lem}
\begin{proof}
It is a direct computation, so we omit the details.
\end{proof}

\begin{thm}\label{fund:thm}
Let $(N,S)$ be a Nijenhuis structure on a \LYP pair $\Big((\g,[\cdot,\cdot],\Courant{\cdot,\cdot,\cdot}),(V;\rho,\mu)\Big)$. Then  a linear deformation can be obtained by setting for any $x,y,z\in \g$
\begin{eqnarray}
\begin{cases}\label{trv1}
\phi(x,y)&=[Nx,y]+[x,Ny]-N[x,y],\\
\varphi_1(x,y,z)&=\Courant{Nx,y,z}+\Courant{x,Ny,z}+\Courant{x,y,Nz}-N\Courant{x,y,z},\\
\varphi_2(x,y,z)&=\Courant{Nx,Ny,z}+\Courant{x,Ny,Nz}+\Courant{Nx,y,Nz}-N\varphi_1(x,y,z),\\
\varrho(x)&=\rho(Nx)+\rho(x)\circ S-S\circ\rho(x),\\
\varpi_1(x,y)&=\mu(Nx,y)+\mu(x,Ny)+\mu(x,y)\circ S-S\circ \mu(x,y),\\
\varpi_2(x,y)&=\mu(Nx,y)\circ S+\mu(x,Ny)\circ S+\mu(Nx,Ny)-S\circ \varpi_1(x,y).
\end{cases}
\end{eqnarray}
Moreover, this deformation  is trivial.
\end{thm}
\begin{proof}
Let $(N,S)$ be a Nijenhuis structure on a \LYP pair $(\g,\rho,\mu)$. Therefore, for any $t$, there holds that
\begin{eqnarray*}
({\Id}_\g+tN)[x,y]_t&=&[({\Id}_\g+tN)x,({\Id}_\g+tN)y],\\
 ({\Id}_\g+tN)\Courant{x,y,z}_t&=&\Courant{({\Id}_\g+tN)x,({\Id}_\g+tN)y,({\Id}_\g+tN)z},\\
  ({\Id}_V+tS)\circ \rho_t(x)&=&\rho\Big(({\Id}_\g+tN)x\Big)\circ ({\Id}_V+tS),\\
   ({\Id}_V+tS)\circ \mu_t(x,y)&=&\mu\Big(({\Id}_\g+tN)x,({\Id}_\g+tN)y\Big)\circ ({\Id}_V+tS),\quad \forall x,y,z\in \g.
\end{eqnarray*}
Then for $t$ sufficiently small, ${\Id}_\g+tN$ and ${\Id}_V+tS$ are isomorphisms between vector spaces. By Lemma \ref{prelem}, we have that $\Big((\g,\br__t,\ltp_t),(V,\rho_t,\mu_t)\Big)$ is a \LYP pair, which means that the linear maps $(\phi,\varphi_i;\varrho,\varpi_i)~(i=1,2)$ given by Eqs. \eqref{trv1} generate a linear deformation. It is obvious that this deformation is trivial. This completes the proof.
\end{proof}

The following proposition demonstrates that a Nijenhuis structure gives rise to a Nijenhuis operator on the semidirect product \LYA.
\begin{pro}\label{pro:Nij}
Let $(N,S)$ be a Nijenhuis structure on a \LYP pair $\Big((\g,[\cdot,\cdot],\Courant{\cdot,\cdot,\cdot}),(V;\rho,\mu)\Big)$. Then $N+S$ is a Nijenhuis operator on the semidirect product Lie-Yamaguti algebra $\g\ltimes V$.
\end{pro}
\begin{proof}
Let $(N,S)$ be a Nijenhuis structure on a \LYP pair $\Big((\g,[\cdot,\cdot],\Courant{\cdot,\cdot,\cdot}),(V;\rho,\mu)\Big)$. Then we have
\begin{eqnarray*}
~ &&[(N+S)(x+u),(N+S)(y+v)]_\ltimes-(N+S)\Big([(N+S)(x+u),y+v]_\ltimes+[x+u,(N+S)(y+v)]_\ltimes\Big)\\
~ &&~~+(N+S)^2[x+u,y+v]_\ltimes\\
~ &=&\Big([Nx,Ny]+\rho(Nx)(Sv)-\rho(Ny)(Su)\Big)-(N+S)\Big([Nx,y]+\rho(Nx)v-\rho(y)(Su)\\
~ &&~~+[x,Ny]+\rho(x)(Sv)-\rho(Ny)u\Big)+(N^2+S^2)\Big([x,y]+\rho(x)v-\rho(y)u\Big)\\
~ &=&\Big([Nx,Ny]-N\big([Nx,y]+[x,Ny]-N[x,y]\big)\Big)+\Big(\rho(Nx)(Sv)-S\big(\rho(Nx)v-\rho(x)(Sv)+S(\rho(x)v)\big)\Big)\\
~ &&-\Big(\rho(Ny)(Su)-S\big(\rho(Ny)u-\rho(y)(Su)+S(\rho(y)u)\big)\Big)\\
~ &=&0.
\end{eqnarray*}
Moreover, we also have
\begin{eqnarray*}
~ &&\Courant{(N+S)(x+u),(N+S)(y+v),(N+S)(z+w)}_\ltimes\\
~ &=&\Courant{Nx,Ny,Nz}+D(Nx,Ny)(Sw)+\mu(Ny,Nz)(Su)-\mu(Nx,Nz)(Sv),
\end{eqnarray*}
and
\begin{eqnarray*}
~ &&(N+S)\Big(\Courant{(N+S)(x+u),(N+S)(y+v),z+w}_\ltimes+\Courant{x+u,(N+S)(y+v),(N+S)(z+w)}_\ltimes\\
~ &&~~+\Courant{(N+S)(x+u),y+v,(N+S)(z+w)}_\ltimes\Big)\\
~ &&~~-(N+S)^2\Big(\Courant{(N+S)(x+u),y+v,z+w}_\ltimes+\Courant{x+u,(N+S)(y+v),z+w}_\ltimes\\
~ &&~~+\Courant{x+u,y+v,(N+S)(z+w)}_\ltimes\Big)+(N+S)^3\Courant{x+u,y+v,z+w}_\ltimes\\
~ &=&(N+S)\Big(\Courant{Nx,Ny,z}+D(Nx,Ny)w+\mu(Ny,z)(Su)-\mu(Nx,z)(Sv)\\
~ &&~~+\Courant{Nx,y,Nz}+D(Nx,y)(Sw)+\mu(y,Nz)(Su)-\mu(Nx,Nz)v\\
~ &&~~+\Courant{x,Ny,Nz}+D(x,Ny)(Sw)+\mu(Ny,Nz)u-\mu(x,Nz)(Sv)\Big)\\
~ &&~~-(N^2+S^2)\Big(\Courant{Nx,y,z}+D(Nx,y)w+\mu(y,z)(Su)-\mu(Nx,z)v\\
~ &&~~+\Courant{x,Ny,z}+D(x,Ny)w+\mu(Ny,z)u-\mu(x,z)(Sv)\\
~ &&~~+\Courant{x,y,Nz}+D(x,y)(Sw)+\mu(y,Nz)u-\mu(x,Nz)v\Big)\\
~ &&~~+(N^3+S^3)\Big(\Courant{x,y,z}+D(x,y)w+\mu(y,z)u-\mu(x,z)v\Big)\\
~ &=&N\Big(\Courant{Nx,Ny,z}+\Courant{Nx,y,Nz}+\Courant{x,Ny,Nz}\Big)\\
~ &&~~-N^2\Big(\Courant{Nx,y,z}+\Courant{x,Ny,z}+\Courant{x,y,Nz}\Big)+N^3\Courant{x,y,z}\\
~ &&~~+S\Big(D(Nx,Ny)w+D(Nx,y)(Sw)+D(x,Ny)(Sw)\Big)\\
~ &&~~-S^2\Big(D(Nx,y)w+D(x,Ny)w+D(x,y)(Sw)\Big)+S^3(D(x,y)w)\\
~ &&~~+S\Big(\mu(Ny,z)(Su)+\mu(y,Nz)(Su)+\mu(Ny,Nz)u\Big)\\
~ &&~~-S^2\Big(\mu(y,z)(Su)+\mu(Ny,z)u+\mu(y,Nz)u\Big)+S^3(\mu(y,z)u)\\
~ &&~~-S\Big(\mu(Nx,z)(Sv)+\mu(Nx,Nz)v+\mu(x,Nz)(Sv)\Big)\\
~ &&~~+S^2\Big(\mu(Nx,z)v+\mu(x,z)(Sv)+\mu(x,Nz)v\Big)-S^3(\mu(x,z)v),
\end{eqnarray*}
which implies that
\begin{eqnarray*}
~ &&\Courant{(N+S)(x+u),(N+S)(y+v),(N+S)(z+w)}_\ltimes\\
~ &=&(N+S)\Big(\Courant{(N+S)(x+u),(N+S)(y+v),z+w}_\ltimes+\Courant{x+u,(N+S)(y+v),(N+S)(z+w)}_\ltimes\\
~ &&~~+\Courant{(N+S)(x+u),y+v,(N+S)(z+w)}_\ltimes\Big)\\
~ &&~~-(N+S)^2\Big(\Courant{(N+S)(x+u),y+v,z+w}_\ltimes+\Courant{x+u,(N+S)(y+v),z+w}_\ltimes\\
~ &&~~+\Courant{x+u,y+v,(N+S)(z+w)}_\ltimes\Big)+(N+S)^3\Courant{x+u,y+v,z+w}_\ltimes.
\end{eqnarray*}
This completes the proof.
\end{proof}

We introduce two linear maps $\hat\rho:\g \to \gl(V)$ and $\hat\mu:\otimes^2\g \to \gl(V)$ by
\begin{eqnarray}
\label{hatrho}\hat\rho(x)&=&\rho(Nx)+\rho(x)\circ S-S\circ \rho(x),\\
\label{hatmu}\hat\mu(x,y)&=&\mu(Nx,y)\circ S+\mu(x,Ny)\circ S+\mu(Nx,Ny)\\
~\nonumber &&-S\Big(\mu(Nx,y)+\mu(x,Ny)+\mu(x,y)\circ S\Big)+S^2\circ \mu(x,y),
\end{eqnarray}
\begin{lem}\label{prolem}
With the above notations, we have
\begin{eqnarray}
\label{hatD}\hat{D}(x,y):= D_{\hat\rho,\hat\mu}&=&D(Nx,y)\circ S+D(x,Ny)\circ S+D(Nx,Ny)\\
~\nonumber &&-S\Big(D(Nx,y)+D(x,Ny)+D(x,y)\circ S\Big)+S^2\circ D(x,y).
\end{eqnarray}
\end{lem}
\begin{proof}
The proof also involves computation.
\end{proof}

\begin{pro}
Let $(N,S)$ be a Nijenhuis structure on a \LYP pair $\Big((\g,[\cdot,\cdot],\Courant{\cdot,\cdot,\cdot}),(V;\rho,\mu)\Big)$. Then $\Big((\g,[\cdot,\cdot]_N,\Courant{\cdot,\cdot,\cdot}_N),(V;\hat\rho,\hat\mu)\Big)$ is a \LYP pair, where $(\g,[\cdot,\cdot]_N,\Courant{\cdot,\cdot,\cdot}_N)$ is the deformed \LYA. Thus $(N,S)$ is a homomorphism from \LYP pair $\Big((\g,[\cdot,\cdot]_N,\Courant{\cdot,\cdot,\cdot}_N),\\
(V;\hat\rho,\hat\mu)\Big)$ to $\Big((\g,[\cdot,\cdot],\Courant{\cdot,\cdot,\cdot}),(V;\rho,\mu)\Big)$.
\end{pro}
\begin{proof}
It was proved in \cite{ShengZhao} that $(\g,\br__N,\ltp_N)$ is a \LYA. What is left is to prove that $(V;\hat\rho,\hat\mu)$ is a representation on $(\g,\br__N,\ltp_N)$. Indeed, by Proposition \ref{pro:Nij}, we have
\begin{eqnarray*}
~ [x+u,y+v]_{N+S}&=&[(N+S)(x+u),y+v]_\ltimes+[x+u,(N+S)(y+v)]_\ltimes-(N+S)[x+u,y+v]_\ltimes\\
~&=&[Nx,y]+[x,Ny]-N[x,y]\\
~ &&+\Big(\rho(Nx)v+\rho(x)(Sv)-S(\rho(x)v)\Big)-\Big(\rho(Ny)u+\rho(y)(Su)-S(\rho(y)u)\Big)\\
~ &=&[x,y]_N+\hat\rho(x)v-\hat\rho(y)u,
\end{eqnarray*}
and by Lemma \ref{prolem}, we also have that
\begin{eqnarray*}
~ &&\Courant{x+u,y+v,z+w}_{N+S}\\
~ &=&\Courant{(N+S)(x+u),(N+S)(y+v),z+w}_\ltimes+\Courant{x+u,(N+S)(y+v),(N+S)(z+w)}_\ltimes\\
~ &&~~+\Courant{(N+S)(x+u),y+v,(N+S)(z+w)}_\ltimes\\
~ &&~~-(N+S)\Big(\Courant{(N+S)(x+u),y+v,z+w}_\ltimes+\Courant{x+u,(N+S)(y+v),z+w}_\ltimes\\
~ &&~~+\Courant{x+u,y+v,(N+S)(z+w)}_\ltimes\Big)+(N+S)^2\Courant{x+u,y+v,z+w}_\ltimes\\
~ &=&\Big(\Courant{Nx,Ny,z}+\Courant{Nx,y,Nz}+\Courant{x,Ny,Nz}\\
~ &&~~-N\big(\Courant{Nx,y,z}+\Courant{x,Ny,z}+\Courant{x,y,Nz}\big)+N^2\Courant{x,y,z}\Big)\\
~ &&+\Big(D(Nx,Ny)w+D(Nx,y)(Sw)+D(x,Ny)(Sw)\\
~ &&~~-S\big(D(Nx,y)w+D(x,Ny)w+D(x,y)(Sw)\big)+S^2(D(x,y)w)\Big)\\
~ &&~~+\Big(\mu(Ny,z)(Su)+\mu(y,Nz)(Su)+\mu(Ny,Nz)u\\
~ &&~~-S\big(\mu(y,z)(Su)+\mu(Ny,z)u+\mu(y,Nz)u\big)+S^2(\mu(y,z)u)\Big)\\
~ &&~~-\Big(\mu(Nx,z)(Sv)+\mu(Nx,Nz)v+\mu(x,Nz)(Sv)\\
~ &&~~+S\big(\mu(Nx,z)v+\mu(x,z)(Sv)+\mu(x,Nz)v\big)-S^3(\mu(x,z)v)\Big)\\
~ &=&\Courant{x,y,z}_N+\hat{D}(x,y)w+\hat\mu(y,z)u-\hat\mu(x,z)v,
\end{eqnarray*}
which implies that $(V;\hat\rho,\hat\mu)$ is a representation on $(\g,\br__N,\ltp_N)$.
\end{proof}

\medskip
\section{Relative Rota-Baxter-Nijenhuis structures and compatible relative Rota-Baxter operators}
\subsection{Relative Rota-Baxter-Nijenhuis structures on \LYP pairs}
In this subsection, we introduce the notion of relative Rota-Baxter-Nijenhuis structures and discuss several properties of relative Rota-Baxter-Nijenhuis structures. First, let us recall some notions in \cite{SZ1} about relative Rota-Baxter operators on \LYA s.
\begin{defi}{\rm (\cite{SZ1})}
Let $(\g,[\cdot,\cdot],\Courant{\cdot,\cdot,\cdot})$ be a Lie-Yamaguti algebra with a representation $(V;\rho,\mu)$ and $T:V\to \g$ a linear map. If $T$ satisfies
\begin{eqnarray*}
[Tu,Tv]&=&T\Big(\rho(Tu)v-\rho(Tv)u\Big),\\
\Courant{Tu,Tv,Tw}&=&T\Big(D(Tu,Tv)w+\mu(Tv,Tw)u-\mu(Tu,Tw)v\Big), \quad\forall u,v,w \in V,
\end{eqnarray*}
then we call $T$ a {\bf relative Rota-Baxter operator} on $\g$ with respect to the representation $(V;\rho,\mu)$.
\end{defi}

\begin{lem}{\rm (\cite{SZ1})}\label{lem:O}
Let $T$ be a relative Rota-Baxter operator on a Lie-Yamaguti algebra $(\g,[\cdot,\cdot],\Courant{\cdot,\cdot,\cdot})$ with respect to $(V;\rho,\mu)$.
Then $(V,[\cdot,\cdot]^T,\Courant{\cdot,\cdot,\cdot}^T)$ is a Lie-Yamaguti algebra, where
\begin{eqnarray}
\label{induce1}[u,v]^T&=&\rho(Tu)v-\rho(Tv)u,\\
\label{induce2}\Courant{u,v,w}^T&=&D_{\rho,\mu}(Tu,Tv)w+\mu(Tv,Tw)u-\mu(Tu,Tw)v,\quad \forall u,v,w\in V.
\end{eqnarray}
Thus $T$ is a Lie-Yamaguti algebra homomorphism from $(V,[\cdot,\cdot]^T,\Courant{\cdot,\cdot,\cdot}^T)$ to $(\g,[\cdot,\cdot],\Courant{\cdot,\cdot,\cdot})$.
Furthermore, $(V,[\cdot,\cdot]^T,\Courant{\cdot,\cdot,\cdot}^T)$ is the {\bf sub-adjacent Lie-Yamaguti algebra} of the pre-Lie-Yamaguti algebra $(V,*^T,\{\cdot,\cdot,\cdot\}^T)$ defined by
\begin{eqnarray}
u*^Tv=\rho(Tu)v,\quad \{u,v,w\}^T=\mu(Tv,Tw)u,\quad \forall u,v, w \in V.\label{preO}
\end{eqnarray}
\end{lem}

With the above notations, let $S\in \gl(V)$ be a linear map on $V$, then the deformed bracket $([\cdot,\cdot]_S^T,\Courant{\cdot,\cdot,\cdot}_S^T)$ of the sub-adjacent Lie-Yamaguti algebra $(V,[\cdot,\cdot]^T,\Courant{\cdot,\cdot,\cdot}^T)$ is given as follows:
\begin{eqnarray*}
[u,v]_S^T&=&[Su,v]^T+[u,Sv]^T-S[u,v]^T,\\
\Courant{u,v,w}_S^T&=&\Courant{Su,Sv,w}^T+\Courant{Su,v,Sw}^T+\Courant{u,Sv,Sw}^T\Big)\\
~ &&-S\Big(\Courant{Su,v,w}^T+\Courant{u,Sv,w}^T+\Courant{u,v,Sw}^T\Big)+S^2\Courant{u,v,w}^T, \quad \forall u,v,w \in V.
\end{eqnarray*}

Moreover, define a pair of brackets $([\cdot,\cdot]_{\hat\rho}^T,\Courant{\cdot,\cdot,\cdot}_{\hat\mu}^T)$ to be
\begin{eqnarray*}
[u,v]_{\hat\rho}^T&=&\hat\rho(Tu)v-\hat\rho(Tv)u,\\
\Courant{u,v,w}_{\hat\mu}^T&=&\hat{D}(Tu,Tv)w+\hat\mu(Tv,Tw)u-\hat\mu(Tu,Tw)v,
\end{eqnarray*}
respectively, where $\hat\rho,~ \hat\mu$ and $\hat D$ are given by \eqref{hatrho}-\eqref{hatD}. Note that in general $([\cdot,\cdot]_{S}^T,~\Courant{\cdot,\cdot,\cdot}_S^T)$, and $([\cdot,\cdot]_{\hat\rho}^T,~\Courant{\cdot,\cdot,\cdot}_{\hat\mu}^T)$ are not necessarily Lie-Yamaguti algebra structures.

\begin{defi}
Let $(N,S)$ be a Nijenhuis structure on a \LYP pair $\Big((\g,[\cdot,\cdot],\Courant{\cdot,\cdot,\cdot}),(V;\rho,\mu)\Big)$ and let $T:V\longrightarrow \g$ be a relative Rota-Baxter operator on $(\g,[\cdot,\cdot],\Courant{\cdot,\cdot,\cdot})$ with respect to  $(V;\rho,\mu)$.
If for all $u,v,w\in V$, the following conditions are satisfied:
\begin{eqnarray}
\label{ON1}N\circ T&=&T\circ S,\\
\label{ON2}~[u,v]_S^T&=&[u,v]^{N\circ T},\\
\label{ON3}\Courant{u,v,w}_S^T&=&\Courant{u,v,w}^{N\circ T}.
\end{eqnarray}
Then we call the triple $(T,S,N)$ a {\bf relative Rota-Baxter-Nijenhuis structure} on the \LYP pair $\Big((\g,[\cdot,\cdot],\Courant{\cdot,\cdot,\cdot}),(V;\rho,\mu)\Big)$. 
\end{defi}
\emptycomment{
It can be deduced from Eqs. \eqref{ON2} and \eqref{ON3} that
\begin{eqnarray*}
~ &&\rho(Tu)(Sv)-\rho(Tv)(Su)=S([u,v]^T),\\
~ &&D(TSu,Tv)(Sw)+D(Tu,TSv)(Sw)+\mu(TSv,Tw)(Su)+\mu(Tv,TSw)(Su)\\
~ &&-\mu(TSu,Tw)(Sv)-\mu(Tu,TSw)(Sv)+S^2(\Courant{u,v,w}^T)\\
~ &=&S\Big(D(TSu,Tv)w+D(Tu,TSv)w+D(Tu,Tv)(Sw)+\mu(Tv,Tw)(Su)+\mu(TSv,Tw)u\\
~ &&+\mu(Tv,TSw)u-\mu(TSu,Tw)v-\mu(Tu,Tw)(Sv)-\mu(Tu,TSw)v\Big).
\end{eqnarray*}

If moreover, a relative Rota-Baxter-Nijenhuis structure $(T,S,N)$ also satisfies the following equation

Then we call $(T,S,N)$ {\bf strong}.}

\begin{lem}
Let $(T,S,N)$ be a relative Rota-Baxter-Nijenhuis structure on a \LYP pair $\Big((\g,[\cdot,\cdot],\Courant{\cdot,\cdot,\cdot}),(V;\rho,\mu)\Big)$. Then we have
\begin{eqnarray*}
~[u,v]_S^T&=&[u,v]_{\hat\rho}^T,\\
\Courant{u,v,w}_S^T&=&\Courant{u,v,w}_{\hat\mu}^T.
\end{eqnarray*}
\end{lem}
\begin{proof}
It follows from \eqref{ON1} and a direct computation.
\end{proof}

Thus if $(T,S,N)$ is a relative Rota-Baxter-Nijenhuis structure on a \LYP pair $\Big((\g,[\cdot,\cdot],\Courant{\cdot,\cdot,\cdot}),\\
(V;\rho,\mu)\Big)$, then the brackets $([\cdot,\cdot]_S^T,\Courant{\cdot,\cdot,\cdot}_S^T)$ and $([\cdot,\cdot]_{\hat\rho}^T,\Courant{\cdot,\cdot,\cdot}_{\hat\mu}^T)$ are equal. Moreover, we have the following conclusion.
\begin{thm}
Let $(T,S,N)$ be a relative Rota-Baxter-Nijenhuis structure on a \LYP pair $\Big((\g,[\cdot,\cdot],\Courant{\cdot,\cdot,\cdot}),(V;\rho,\mu)\Big)$. Then $S$ is a Nijenhuis operator on the sub-adjacent Lie-Yamaguti algebra $(V,[\cdot,\cdot]^T,\Courant{\cdot,\cdot,\cdot}^T)$. Thus the brackets $([\cdot,\cdot]_{S}^T,~\Courant{\cdot,\cdot,\cdot}_S^T)$ and $([\cdot,\cdot]_{\hat\rho}^T,~\Courant{\cdot,\cdot,\cdot}_{\hat\mu}^T)$ are the Lie-Yamaguti algebra structures on $V$.
\end{thm}
\begin{proof}
Since $(T,S,N)$ is a relative Rota-Baxter-Nijenhuis structure, then we have
\begin{eqnarray*}
~ &&[Su,Sv]^T\stackrel{\eqref{induce1}}=\rho(TSu)(Sv)-\rho(TSv)(Su)\\
~ &\stackrel{\eqref{ON1},\eqref{Nijstru1}}=&S\Big(\rho(TSu)v+\rho(Tu)(Sv)\Big)-S^2\Big(\rho(Tu)v\Big)\\
~ &&-S\Big(\rho(TSv)u+\rho(Tv)(Su)\Big)+S^2\Big(\rho(Tv)u\Big)\\
~ &\stackrel{\eqref{induce1}}=&S\Big([Su,v]^T+[u,Sv]^T\Big)-S^2[u,v]^T,
\end{eqnarray*}
and
\begin{eqnarray*}
~ &&\Courant{Su,Sv,Sw}^T\\
&\stackrel{\eqref{induce2}}=&D(TSu,TSv)(Sw)+\mu(TSv,TSw)(Su)-\mu(TSu,TSw)(Sv)\\
~ &\stackrel{\eqref{ON1},\eqref{Nijstru2},\eqref{Nijstru3}}=&S\Big(D(TSu,Tv)(Sw)+D(Tu,TSv)(Sw)+D(TSu,TSv)w\Big)\\
~ &&-S^2\Big(D(TSu,Tv)w+D(Tu,TSv)w+D(Tu,Tv)(Sw)\Big)+S^3(D(Tu,Tv)w)\\
~ &&+S\Big(\mu(TSv,Tw)(Su)+\mu(Tv,TSw)(Su)+\mu(TSv,TSw)u\Big)\\
~ &&-S^2\Big(\mu(TSv,Tw)u+\mu(Tv,TSw)u+\mu(Tv,Tw)(Su)\Big)+S^3(\mu(Tv,Tw)u)\\
 &&-S\Big(\mu(TSu,Tw)(Sv)+\mu(Tu,TSw)(Sv)+\mu(TSu,TSw)v\Big)\\
~ &&+S^2\Big(\mu(TSu,Tw)v+\mu(Tu,TSw)v+\mu(Tu,Tw)(Sv)\Big)-S^3(\mu(Tu,Tw)v)\\
~ &\stackrel{\eqref{induce2}}=&S\Big(\Courant{Su,Sv,w}^T+\Courant{Su,v,Sw}^T+\Courant{u,Sv,Sw}^T\Big)\\
~ &&-S^2\Big(\Courant{Su,v,w}^T+\Courant{u,Sv,w}^T+\Courant{Su,v,Sw}^T\Big)+S^3\Courant{u,v,w}^T.
\end{eqnarray*}
Thus, $S$ is a Nijenhuis operator on the Lie-Yamaguti algebra $(V,[\cdot,\cdot]^T,\Courant{\cdot,\cdot,\cdot}^T)$.
\end{proof}

\begin{thm}\label{Nij}
Let $(T,S,N)$ be a relative Rota-Baxter-Nijenhuis structure on a \LYP pair $\Big((\g,[\cdot,\cdot],\Courant{\cdot,\cdot,\cdot}),(V;\rho,\mu)\Big)$. Then we have
\begin{itemize}
\item [\rm (i)] $T$ is a relative Rota-Baxter operator on the deformed Lie-Yamaguti algebra $(\g,[\cdot,\cdot]_N,\Courant{\cdot,\cdot,\cdot}_N)$ with respect to a representation $(V;\hat\rho,\hat\mu)$;
\item [\rm (ii)] $N\circ T$ is a relative Rota-Baxter operator on the Lie-Yamaguti algebra $(\g,[\cdot,\cdot],\Courant{\cdot,\cdot,\cdot})$ with respect to a representation $(V;\rho,\mu)$.
\end{itemize}
\end{thm}
\begin{proof}
Since $T$ is a relative Rota-Baxter operator on the Lie-Yamaguti algebra $(\g,[\cdot,\cdot],\Courant{\cdot,\cdot,\cdot})$ with respect to a representation $(V;\rho,\mu)$ and $N\circ T=T\circ S$, we have
\begin{eqnarray*}
~ T[u,v]_{\hat\rho}^T&=&T[u,v]_S^T=T\Big([Su,v]^T+[u,Sv]^T-S[u,v]^T\Big)\\
~ &=&[TSu,Tv]+[Tu,TSv]-T\circ S[u,v]^T\\
~ &=&[NTu,Tv]+[Tu,NTv]-N\circ T[u,v]^T\\\
~ &=&[NTu,Tv]+[Tu,NTv]-N[Tu,Tv]\\
~ &=&[Tu,Tv]_N,
\end{eqnarray*}
and
\begin{eqnarray*}
~&& T\Courant{u,v,w}_{\hat\mu}^T=T\Courant{u,v,w}_S^T\\
~ &=&T\Big(\Courant{Su,Sv,w}^T+\Courant{Su,v,Sw}^T+\Courant{u,Sv,Sw}^T\\
~ &&-S\Courant{Su,v,w}^T-S\Courant{u,Sv,w}^T-S\Courant{u,v,Sw}^T+S^2\Courant{u,v,w}^T\Big)\\
~ &=&\Courant{TSu,TSv,w}+\Courant{TSu,Tv,TSw}+\Courant{Tu,TSv,TSw}\\
~ &&-T\circ S\Courant{Su,TSv,w}^T-T\circ S\Courant{Su,Sv,w}-T\circ S\Courant{u,v,Sw}^T+T\circ S^2\Courant{u,v,w}^T\\
~ &=&\Courant{NTu,NTv,w}+\Courant{NTu,Tv,NTw}+\Courant{Tu,NTv,NTw}\\
~ &&-N\Big(\Courant{NTu,Tv,Tw}^T+\Courant{NTu,Tv,NTw}+\Courant{Tu,Tv,NTw}^T\Big)+N^2\Courant{Tu,Tv,Tw}\\
~ &=&\Courant{Tu,Tv,Tw}_N.
\end{eqnarray*}
This proves (i).
By \eqref{ON2} and \eqref{ON3}, we have that
$$N\circ T[u,v]^{N\circ T}=N\circ T[u,v]_S^T=N[Tu,Tv]_N=[NTu,NTv],$$
and
$$N\circ T\Courant{u,v,w}^{N\circ T}=N\circ T\Courant{u,v,w}_S^T=N\Courant{Tu,Tv,Tw}_N=\Courant{NTu,NTv,NT.w}$$
This proves (ii).
\end{proof}
\emptycomment{
Let $(T,S,N)$ be a relative Rota-Baxter-Nijenhuis structure on a \LYP pair $\Big((\g,\br,\ltp),(V;\rho,\mu)\Big)$. We give the following commutative diagram to summarize the conclusions in the former section, each of which arrows is a \LYA~ homomorphism:

\[
 \xymatrix{
 (V,[\cdot,\cdot]_S^T,\Courant{\cdot,\cdot,\cdot}_S^T)\ar[d]^{S}\ar[dr]^{N \circ T} \ar[r]^{T} & (\g,[\cdot,\cdot]_N,\Courant{\cdot,\cdot,\cdot}_N)\ar[d]^{N}    \\
(V,[\cdot,\cdot]^T,\Courant{\cdot,\cdot,\cdot}^T) \ar[r]^{T} &(\g,\br,\ltp)
}
\]
}

\subsection{Compatible relative Rota-Baxter operators on \LYA s}
In this subsection, we introduce the notion of compatible relative Rota-Baxter operators on \LYA, and we show that a relative Rota-Baxter-Nijenhuis structure gives rise to a pair of relative Rota-Baxter operators on \LYA s under a certain condition.

\begin{defi}\label{compO}
Let $T_1,~T_2: V \to \g$ be two relative Rota-Baxter operators on a \LYA~ $(\g,\br,\ltp)$ with respect to a representation $(V;\rho,\mu)$. If for all $k_1$ and $k_2$, $k_1T_1+k_2T_2$ is still a relative Rota-Baxter operator on $\g$, then we say that $T_1$ and $T_2$ are {\bf compatible}.
\end{defi}

By a direct computation, we have the following proposition.
\begin{pro}\label{compatible}
Let $T_1,~T_2:V \to \g$ be two relative Rota-Baxter operators on a \LYA~ $(\g,\br,\ltp)$. Then $T_1$ and $T_2$ are compatible if and only if the following equations are satisfied:
\begin{eqnarray*}
~ &&[T_iu,T_jv]+[T_ju,T_iv]\\
~ &=&T_i\Big(\rho(T_ju)v-\rho(T_jv)u\Big)+T_j\Big(\rho(T_iu)v-\rho(T_iv)u\Big),\\
~ &&\Courant{T_iu,T_iv,T_jw}+\Courant{T_iu,T_jv,T_iw}+\Courant{T_ju,T_iv,T_iw}\\
~ &=&T_i\Big(D(T_iu,T_jv)w+\mu(T_iv,T_jw)u-\mu(T_iu,T_jw)v\\
~ &&+D(T_ju,T_iv)w+\mu(T_jv,T_iw)u-\mu(T_ju,T_iw)v\Big)\\
~ &&+T_j\Big(D(T_iu,T_iv)w+\mu(T_iv,T_iw)u-\mu(T_iu,T_iw)v\Big), \quad(i,j=1,2).
\end{eqnarray*}
\end{pro}

The following proposition shows that a pair of compatible relative Rota-Baxter operators with an invertible condition yields a Nijenhuis operator.
\begin{pro}
Let $T_1,~T_2:V \to \g$ be two compatible relative Rota-Baxter operators on a \LYA~ $(\g,\br,,\ltp)$ with respect to a representation $(V;\rho,\mu)$. If $T_2$ is invertible, then $N=T_1 \circ T_2^{-1}$ is a Nijenhuis operator on the \LYA ~$(\g,\br,,\ltp)$.
\end{pro}
\begin{proof}
Since $T_2$ is invertible, for all $x,y,z\in\g$, there exist $u,v,w\in V$, such that $x=T_2u,~y=T_2v,~z=T_2w$. Thus it is sufficient to show that
{\footnotesize
\begin{eqnarray}
~\label{NIJE1}[NT_2u,NT_2v]&=&N\Big([NT_2u,T_2v]+[T_2u,NT_2v]\Big)-N^2[T_2u,T_2v],\\
~\label{NIJE2}\Courant{NT_2u,NT_2v,NT_2w}&=&N\Big(\Courant{NT_2u,NT_2v,T_2w}+\Courant{NT_2u,T_2v,NT_2w}+\Courant{T_2u,NT_2v,NT_2w}\Big)\\
~ \nonumber&&-N^2\Big(\Courant{NT_2u,T_2v,T_2w}+\Courant{T_2u,NT_2v,T_2w}+\Courant{T_2u,T_2v,NT_2w}\Big)\\
~ \nonumber &&+N^3\Courant{T_2u,T_2v,T_2w}.
\end{eqnarray}}
Since $T_1=N\circ T_2$ is a relative Rota-Baxter operator, the left hand side of Eqs. \eqref{NIJE1} and \eqref{NIJE2} are
$$NT_2\Big(\rho(NT_2u)v-\rho(NT_2v)u\Big)$$
and
$$NT_2\Big(D(NT_2u,NT_2v)w+\mu(NT_2v,NT_2w)u-\mu(NT_2u,NT_2w)v\Big)$$
respectively.
Since $T_1=N\circ T_2$ and $T_2$ are compatible, by Proposition \ref{compatible}, we have
\begin{eqnarray}
~\nonumber &&[NT_2u,T_2v]+[T_2u,NT_2v]\\
~ \label{N}&=&T_2\Big(\rho(NT_2u)v-\rho(NT_2v)u\Big)+NT_2\Big(\rho(T_2u)v-\rho(T_2v)u\Big)\\
~ \nonumber&=&T_2\Big(\rho(NT_2u)v-\rho(NT_2v)u\Big)+N[T_2u,T_2v].
\end{eqnarray}
The action of $N$ on both sides of \eqref{N} yields Eq \eqref{NIJE1}. On the other hand, as for \eqref{NIJE2} we have 
\begin{eqnarray*}
~ &&\Courant{NT_2u,NT_2v,T_2w}+\Courant{T_2u,NT_2v,NT_2w}+\Courant{NT_2u,T_2v,NT_2w}\\
~ &=&T_2\Big(D(NT_2u,NT_2v)w+\mu(NT_2v,NT_2w)u-\mu(NT_2u,NT_2w)v\Big)\\
~ &&+NT_2\Big(D(NT_2u,T_2v)w+\mu(NT_2v,T_2w)u-\mu(NT_2u,T_2w)v\\
~ &&+D(T_2u,NT_2v)w+\mu(T_2v,NT_2w)u-\mu(T_2u,NT_2w)v\Big),
\end{eqnarray*}
and
\begin{eqnarray*}
~ &&\Courant{NT_2u,T_2v,T_2w}+\Courant{T_2u,NT_2v,T_2w}+\Courant{T_2u,T_2v,NT_2w}\\
~ &=&NT_2\Big(D(T_2u,T_2v)w+\mu(T_2v,T_2w)u-\mu(T_2u,T_2w)v\Big)\\
~ &&+T_2\Big(D(NT_2u,T_2v)w+\mu(NT_2v,T_2w)u-\mu(NT_2u,T_2w)v\\
~ &&+D(T_2u,NT_2v)w+\mu(T_2v,NT_2w)u-\mu(T_2u,NT_2w)v\Big).
\end{eqnarray*}
Thus the right hand side of Eq \eqref{NIJE2} is
$$NT_2\Big(D(NT_2u,NT_2v)w+\mu(NT_2v,NT_2w)u-\mu(NT_2u,NT_2w)v\Big),$$
which is exactly the left hand side of \eqref{NIJE2}. This finishes the proof.
\end{proof}

Given a relative Rota-Baxter-Nijenhuis structure $(T,S,N)$, by Theorem \ref{Nij}, $T \circ S=N\circ T$ is a relative Rota-Baxter operator. However, in general, $T$ and $T\circ S$ are not compatible spontaneously. Indeed, we can show that they are compatible under a certain condition.
\begin{pro}\label{pro:Nij}
Let $(T,S,N)$ be a relative Rota-Baxter-Nijenhuis structure on a \LYP
pair $\Big((\g,[\cdot,\cdot],\Courant{\cdot,\cdot,\cdot}),(V;\rho,\mu)\Big)$ satisfying
\begin{equation}
	D(Tu,Tv)(Sw)+\mu(Tv,Tw)(Su)-\mu(Tu,Tw)(Sv)=S\big(\Courant{u,v,w}^T\big), \quad \forall u,v,w \in \g.\label{strong2}
\end{equation}
Then $T$ and $T \circ S$ are compatible relative Rota-Baxter operators.
\end{pro}
\begin{proof}
Since $T$ is a relative Rota-Baxter operator, we have
\begin{eqnarray*}
~ &&[TSu,Tv]+[Tu,TSv]\\
~ &=&T\Big(\rho(TSu)v-\rho(Tv)(Su)+\rho(Tu)(Sv)-\rho(TSv)u\Big)\\
~ &\stackrel{\eqref{ON2}}=&T\Big(\rho(TSu)v-\rho(TSv)u\Big)+TS\Big(\rho(Tu)v-\rho(Tv)u\Big).
\end{eqnarray*}
Moreover, Eqs. \eqref{ON3} and \eqref{strong2} imply that
\begin{eqnarray}
\nonumber && D(TSu,Tv)(Sw)+\mu(TSv,Tw)(Su)-\mu(TSu,Tw)(Sv)\\
\nonumber&&+D(Tu,TSv)(Sw)+\mu(Tv,TSw)(Su)-\mu(Tu,TSw)(Sv)\\
\label{strong1}&=&S\Big(D(TSu,Tv)w+\mu(TSv,Tw)u-\mu(TSu,Tw)v\\
\nonumber&&+D(Tu,TSv)w+\mu(Tv,TSw)u-\mu(Tu,TSw)v\Big).
\end{eqnarray}
By the fact that $T$ is a relative Rota-Baxter operator again, we have
\begin{eqnarray*}
&&\Courant{TSu,TSv,Tw}+\Courant{TSu,Tv,TSw}+\Courant{Tu,TSv,TSw}\\
&=&T\Big(D(TSu,TSv)w+\mu(TSv,Tw)(Su)-\mu(TSu,Tw)(Sv)\\
 &&+D(TSu,Tv)(Sw)+\mu(Tv,TSw)(Su)-\mu(TSu,TSw)v\\
&&+D(Tu,TSv)(Sw)+\mu(TSv,TSw)u-\mu(Tu,TSw)(Sv)\Big)\\
&\stackrel{\eqref{strong1}}=&T\Big(D(TSu,TSv)w+\mu(TSv,TSw)u-\mu(TSu,TSw)v\Big)\\
&&+T\circ S\Big(D(TSu,Tv)w+\mu(TSv,Tw)u-\mu(TSu,Tw)v\\
&&+D(Tu,TSv)w+\mu(Tv,TSw)u-\mu(Tu,TSw)v\Big).
\end{eqnarray*}
Similarly, we deduce from \eqref{strong2} that
\begin{eqnarray*}
~ &&\Courant{Tu,TSv,Tw}+\Courant{Tu,Tv,TSw}+\Courant{TSu,Tv,Tw}\\
~ &=&T\circ S\Big(D(Tu,Tv)w+\mu(Tv,Tw)u-\mu(Tu,Tw)v\Big)\\
~ &&+T\Big(\rho(TSu,Tv)w+\mu(TSv,Tw)u-\mu(TSu,Tw)v+D(Tu,TSv)w\\
~ &&+\mu(Tv,TSw)u-\mu(Tu,TSw)v\Big).
\end{eqnarray*}
By Proposition \ref{compatible}, the conclusion thus follows.
\end{proof}

\begin{rmk}
In the context of Lie algebras or associative algebras, for any $k\in \mathbb N$, $T_k=N^k\circ T$ are relative Rota-Baxter operators and what is more, $T_k$ and $T_l$ are compatible. However, this property fails in the context of $3$-Lie algebras (\cite{Zhao-Liu}) and \LYA s, which leads to a different content of the relative Rota-Baxter-Nijenhuis structures on ternary case.
\end{rmk}

\section{$r$-matrix-Nijenhuis structures and Rota-Baxter-Nijenhuis structures}
In this section, we introduce the notions of $r$-matrix-Nijenhuis structures and Rota-Baxter-Nijenhuis structures on \LYA s and investigate the relationship between them.
If $(V;\rho,\mu)$ is a representation of a \LYA ~$(\g,\br,,\ltp)$, then the dual representation of $(V;\rho,\mu)$ is given by
$$\Big(V^*;\rho^*,-\mu^*\tau\Big),$$
where $\rho^*,~\mu^*$ are dual to $-\rho,~-\mu$ respectively, and $\tau:\otimes^2\g\to \otimes^2\g$ is the switching map, i.e.,
$$\tau(x\otimes y)=y\otimes x,\quad \forall x\otimes y\in \otimes^2\g.$$
See \cite{SZ1} for ore details about dual representation of \LYA s.

\begin{ex}{\rm (\cite{SZ1})}
Let $(\g;\ad,\frkR)$ be the adjoint representation of a \LYA ~$\g$ given in Example \ref{ad}, then $(\g^*;\ad^*,-\frkR^*\tau)$ is the dual representation of $(\g;\ad,\frkR)$, which is called the {\bf coadjoint representation}.
\end{ex}

First, we introduce the notions of dual Nijenhuis structures and relative Rota-Baxter-dual-Nijenhuis structures on \LYA s.

\begin{defi}
Let $\Big((\g,\br,,\ltp),(V;\rho,\mu)\Big)$ be a \LYP pair, $N\in \gl(\g)$, $S\in \gl(V)$, and $T:V\longrightarrow \g$ a relative Rota-Baxter operator on $\g$ with respect to $(V;\rho,\mu)$.
\begin{itemize}
\item [\rm (i)]  A pair $(N,S)$ is called a {\bf dual Nijenhuis structure} if $(N,S^*)$ is a Nijenhuis structure on the \LYP pair $\Big((\g,\br,,\ltp),(V^*;\rho^*,-\mu^*\tau)\Big)$.
\item [\rm (ii)] If moreover, $T$ and dual Nijenhuis structure $(N,S)$ satifies Eqs. \eqref{ON1}-\eqref{ON3}, the triple $(T,S,N)$ is called a {\bf relative Rota-Baxter-dual-Nijenhuis structure} on \LYP pair $\Big((\g,\br,,\ltp),(V;\rho,\mu)\Big)$.
\end{itemize}
\end{defi}

Note that a pair $(N,S)$ is a dual Nijenhuis structure is equivalent to that $N$ is a Nijenhuis operator on $(\g,\br,,\ltp)$ and for all $x,y\in \g,v\in V$, the following conditions are satisfied:
\begin{eqnarray*}
\rho(Nx)(Sv)&=&S\big(\rho(Nx)v-\rho(x)(Sv)\Big)+\rho(x)S^2v,\\
\mu(Nx,Ny)(Sv)&=&S\Big(\mu(Nx,Ny)v-\mu(Nx,y)(Sv)-\mu(x,Ny)(Sv)+\mu(x,y)(S^2v)\Big)\\
~ &&+\mu(Nx,y)(S^2v)+\mu(x,Ny)(S^2v)-\mu(x,y)(S^3v).
\end{eqnarray*}

\begin{ex}
Let $N$ be a Nijenhuis operator on a \LYA ~$(\g,\br,,\ltp)$. Then the pair $(N,N^*)$ is a dual Nijenhuis structure on the \LYP pair $\Big((\g,\br,,\ltp),(\g^*;\ad^*,-\frkR^*\tau)\Big)$.
\end{ex}

\begin{rmk}
Note that in the context of Lie algebras or associative algebras, a relative Rota-Baxter-Nijenhuis structure $(T,N,S)$ is also a relative Rota-Baxter-dual-Nijenhuis structure under the condition that $T$ is invertible. However, this property fails in the the context of $3$-Lie algebras or \LYA s even though $T$ is invertible. This is another different contents of ternary operations.
\end{rmk}

The notion of the classical Lie-Yamaguti $r$-matrices was introduced in \cite{ZQ2}, and we call it an $r$-matrix in the sequel. Moreover, we proved that a skew-symmetric $2$-tensor $\pi\in \wedge^2\g$ is an $r$-matrix if and only if the induced map $\pi^\sharp:\g^*\longrightarrow\g$ is a relative Rota-Baxter operator with respect to the coadjoint representation, where $\pi^\sharp$ is defined to be
\begin{eqnarray}
\pair{\pi^\sharp(\alpha),\beta}=\pi(\alpha,\beta),\quad\forall \alpha,\beta\in \g^*.\label{pisharp}
\end{eqnarray}

In the following, we introduce the notion of $r$-matrix-Nijenhuis structures on \LYA s.
\begin{defi}
Let $(\g,\br,,\ltp)$ be a \LYA, $\pi\in \wedge^2\g$ an $r$-matrix, and $N:\g\longrightarrow \g$ a Nijenhuis operator on $\g$. A pair $(\pi,N)$ is called an {\bf $r$-matrix-Nijenhuis structure} on $\g$ if for all $\alpha,\beta,\gamma\in \g^*$, they satisfy
\begin{eqnarray}
N\circ \pi^\sharp&=&\pi^\sharp\circ N^*,\\
~[\alpha,\beta]^{N\circ \pi^\sharp}&=&[\alpha,\beta]^{\pi^\sharp}_{N^*},\\
~\Courant{\alpha,\beta,\gamma}^{N\circ \pi^\sharp}&=&\Courant{\alpha,\beta,\gamma}^{\pi^\sharp}_{N^*},
\end{eqnarray}
where $\pi^\sharp$ is given by \eqref{pisharp}.
\end{defi}

\begin{rmk}
 It is easy to see that an $r$-matrix-Nijenhuis structure $(\pi,N)$ on $\g$ is a special relative Rota-Baxter-dual-Nijenhuis structure on the \LYP pair $\Big((\g,\br,,\ltp),(\g^*,\ad^*,-\frkR^*\tau)\Big)$:
 $$\Big(T=\pi^\sharp,S=N^*,N\Big).$$
\end{rmk}
Similar to the case of $r$-matrix-Nijenhuis structures, we give the notion of Rota-Baxter-Nijenhuis structures as follows.

\begin{defi}
Let $(\g,\br,,\ltp)$ be a \LYA. Suppose that $R:\g\to \g$ is a Rota-Baxter operator and $N:\g\to \g$ is a Nijenhuis operator on $\g$. A pair $(R,N)$ is a {\bf Rota-Baxter-Nijenhuis structure} on $\g$ if for all $x,y,z\in \g$, they satisfy
\begin{eqnarray}
N\circ R&=&R\circ N,\label{rn1}\\
~[x,y]^{N\circ R}&=&[x,y]^{R}_{N},\label{rn2}\\
~\Courant{x,y,z}^{N\circ R}&=&\Courant{x,y,z}^{R}_{N},\label{rn3}
\end{eqnarray}
\end{defi}

\begin{rmk}
 It is not hard to see that a Rota-Baxter-Nijenhuis structure $(R,N)$ on $\g$ is a special relative Rota-Baxter-Nijenhuis structure on the \LYP pair $\Big((\g,\br,,\ltp),(\g;\ad,\frkR)\Big)$:
$$\Big(T=R,S=N,N\Big).$$
\end{rmk}

\begin{defi}(\cite{Kikkawa})
A {\bf quadratic Lie-Yamaguti algebra} is a pair $\Big((\g,\br,,\ltp),B\Big)$, where $(\g,\br,,\ltp)$ is a Lie-Yamaguti algebra, and $B \in \otimes^2\g^*$ is a nondegenerate symmetric bilinear form satisfying the following invariant conditions
\begin{eqnarray*}
\label{invr1}B([x,y],z)&=&-B(y,[x,z]),\\
\label{invr2}B(\Courant{x,y,z},w)&=&B(x,\Courant{w,z,y}), \quad \forall x,y,z \in \g.
\end{eqnarray*}
\end{defi}

Then $B$ induces a bijection $B^\sharp:\g^*\to \g$ defined by
$$B(x,y)=\langle(B^\sharp)^{-1}(x),y\rangle,\quad \forall x,y\in\g.$$

In \cite{SZ1}, by the invariance of $B$, it is proved that
\begin{eqnarray}
\label{inva1}B^\sharp\Big(\ad^*_x\alpha\Big)&=&\ad_x\Big(B^\sharp(\alpha)\Big),\label{inv1}\\
~\label{inva2}-B^\sharp\Big(\frkR^*(y,x)\alpha\Big)&=&\frkR(x,y)\Big(B^\sharp(\alpha)\Big),\quad \forall \alpha\in \g^*.\label{inv2}
\end{eqnarray}

Moreover, it is easy to see that
\begin{eqnarray}
\label{inva3}B^\sharp\Big(\frkL^*(x,y)\alpha\Big)&=&\frkL(x,y)\Big(B^\sharp(\alpha)\Big),\quad \forall \alpha\in \g^*.\label{inv3}
\end{eqnarray}

A {\bf skew-symmetric endomorphism} of $(\g,B)$ is a linear map $R:\g\to \g$ such that $R\circ B^\sharp:\g^*\to \g$ is skew-symmetric in the sense that
\begin{eqnarray*}
 \langle\alpha,R\circ B^\sharp(\beta)\rangle+\langle\beta,R\circ B^\sharp(\alpha)\rangle=0,\quad \forall \alpha, \beta \in \g^*,\label{skew}
 \end{eqnarray*}

The following theorem illustrates the relationship between $r$-matrix-Nijenhuis structures and Rota-Baxter-Nijenhuis structures.

\begin{thm}
 Let $\Big((\g,\br,,\ltp),B\Big)$ be a quadratic \LYA, $N:\g \longrightarrow \g$ a Nijenhuis operator on $\g$ and $R:\g\longrightarrow\g$ a skew-symmetric endomorphism of $(\g,B)$. Set $\pi^\sharp=R\circ B^\sharp$ and assume that $B$ and $N$ are compatible, i.e.,
\begin{eqnarray}
B^\sharp\circ N^*=N\circ B^\sharp.\label{app}
\end{eqnarray}
Then we have
\begin{itemize}
\item [\rm (i)] If $(R,N)$ is a Rota-Baxter-Nijenhuis structure on $\g$, then $(\pi,N)$ is an $r$-matrix-Nijenhuis structure on $\g$.
\item [\rm (ii)] If $(\pi,N)$ is an $r$-matrix-Nijenhuis structure on the \LYA~ $\g$, then the pair $(R=\pi^\sharp\circ (B^\sharp)^{-1},N)$ is a Rota-Baxter-Nijenhuis structure on $\g$.
\end{itemize}
\end{thm}
\begin{proof}
First, by \eqref{rn1} and \eqref{app}, it is sasy to see that $N\circ \pi^\sharp=\pi^\sharp\circ N^*$.
Next, we show that $\pi^\sharp:=R\circ B^\sharp$ is a relative Rota-Baxter operator with respect to the coadjoint representation, which is equivalent to that $\pi\in \wedge^2\g$ is an $r$-matrix. Indeed, by the invariance of $B$, we have
\begin{eqnarray*}
~ &&\pi^\sharp\Big(\ad^*_{\pi^\sharp(\alpha)}\beta-\ad^*_{\pi^\sharp(\beta)}\alpha\Big)\\
~ &=&R\circ B^\sharp\Big(\ad^*_{\pi^\sharp(\alpha)}\beta-\ad^*_{\pi^\sharp(\beta)}\alpha\Big)\\
~ &\stackrel{\eqref{inva1}}=&R\Big(\ad_{\pi^\sharp(\alpha)}B^\sharp(\beta)-\ad_{\pi^\sharp(\beta)}B^\sharp(\alpha)\Big)\\
~ &=&R\Big([R\circ B^\sharp(\alpha),B^\sharp(\beta)]+[B^\sharp(\alpha),R\circ B^\sharp(\beta)]\Big)\\
~ &=&[R\circ B^\sharp(\alpha),R\circ B^\sharp(\beta)]=[\pi^\sharp(\alpha),\pi^\sharp(\beta)],
\end{eqnarray*}
and similarly
\begin{eqnarray*}
~ &&\pi^\sharp\Big(\frkL^*(\pi^\sharp(\alpha),\pi^\sharp(\beta))\gamma-\frkR^*(\pi^\sharp(\gamma),\pi^\sharp(\beta))\alpha
+\frkR^*(\pi^\sharp(\gamma),\pi^\sharp(\alpha))\beta\Big)\\
~ &=&R\Big(\Courant{\pi^\sharp(\alpha),\pi^\sharp(\beta),B^\sharp(\gamma)}+\Courant{B^\sharp(\alpha),\pi^\sharp(\beta),\pi^\sharp(\gamma)}
-\Courant{B^\sharp(\beta),\pi^\sharp(\alpha),\pi^\sharp(\gamma)}\Big)\\
~ &=&\Courant{R\circ B^\sharp(\alpha),R\circ B^\sharp(\beta),R\circ B^\sharp(\gamma)}=\Courant{\pi^\sharp(\alpha),\pi^\sharp(\beta),\pi^\sharp(\gamma)}.
\end{eqnarray*}
What is left is to show $[\alpha,\beta]^{N\circ \pi^\sharp}=[\alpha,\beta]^{\pi^\sharp}_{N^*}$ and $\Courant{\alpha,\beta,\gamma}^{N\circ \pi^\sharp}=\Courant{\alpha,\beta,\gamma}^{\pi^\sharp}_{N^*}$.

Let $\alpha=(B^\sharp)^{-1}(x),~\beta=(B^\sharp)^{-1}(y),~\gamma=(B^\sharp)^{-1}(z)$, by \eqref{inv1}-\eqref{inv3}, we have
\begin{eqnarray*}
~[\alpha,\beta]^{\pi^\sharp}&=&\ad^*_{\pi^\sharp((B^\sharp)^{-1}(x))}(B^\sharp)^{-1}(y)-\ad^*_{\pi^\sharp((B^\sharp)^{-1}(y))}(B^\sharp)^{-1}(x)\\
~ &\stackrel{\eqref{inva1}}=&\ad^*_{R(x)}(B^\sharp)^{-1}(y)-\ad^*_{R(y)}(B^\sharp)^{-1}(x)\\
~ &=&(B^\sharp)^{-1}\Big([R(x),y]-[R(y),x]\Big)=(B^\sharp)^{-1}[x,y]^R,
\end{eqnarray*}
and
\begin{eqnarray*}
~ &&\Courant{\alpha,\beta,\gamma}^{\pi^{\sharp}}\\
~ &=&\frkL^*\Big(\pi^\sharp(B^\sharp)^{-1}(x),\pi^\sharp(B^\sharp)^{-1}(y)\Big)\Big((B^\sharp)^{-1}(z)\Big)
-\frkR^*\Big(\pi^\sharp(B^\sharp)^{-1}(z),\pi^\sharp(B^\sharp)^{-1}(y)\Big)\Big((B^\sharp)^{-1}(x)\Big)\\
~ &&+\frkR^*\Big((\pi^\sharp(B^\sharp)^{-1}(z),\pi^\sharp(B^\sharp)^{-1}(x)\Big)\Big((B^\sharp)^{-1}(y)\Big)\\
~ &=&\frkL^*\Big(R(x),R(y)\Big)\Big((B^\sharp)^{-1}(z)\Big)-\frkR^*\Big(R(z),R(y)\Big)\Big((B^\sharp)^{-1}(x)\Big)
+\frkR^*\Big(R(z),R(x)\Big)\Big((B^\sharp)^{-1}(y)\Big)\\
~ &\stackrel{\eqref{inva2},\eqref{inva3}}=&(B^\sharp)^{-1}\Big(\Courant{R(x),R(y),z}+\Courant{x,R(y),R(z)}-\Courant{y,R(x),R(z)}\Big)=(B^\sharp)^{-1}\Courant{x,y,z}^R.
\end{eqnarray*}
Thus we have
\begin{eqnarray}
~[(B^\sharp)^{-1}(x),(B^\sharp)^{-1}(y)]^{\pi^\sharp}&=&(B^\sharp)^{-1}[x,y]^R,\label{con1}\\
\Courant{(B^\sharp)^{-1}(x),(B^\sharp)^{-1}(y),(B^\sharp)^{-1}(z)}^{\pi^{\sharp}}&=&(B^\sharp)^{-1}\Courant{x,y,z}^R.\label{con2}
\end{eqnarray}
Hence, by \eqref{app}, \eqref{con1}, and \eqref{con2},  we have
\begin{eqnarray*}
~ &&[\alpha,\beta]^{N\circ \pi^\sharp}-[\alpha,\beta]^{\pi^\sharp}_{N^*}
=[\alpha,\beta]^{N\circ \pi^\sharp}-\Big([N^*(\alpha),\beta]^{\pi^\sharp}+[\alpha,N^*(\beta)]^{\pi^\sharp}-N^*[\alpha,\beta]^{\pi^\sharp}\Big)\\
~ &\stackrel{\eqref{con1}}=&(B^\sharp)^{-1}\Big([x,y]^{N\circ R}\Big)-\Big([N^*((B^{\sharp})^{-1}(x)),(B^{\sharp})^{-1}(y)]^{\pi^\sharp}
+[(B^{\sharp})^{-1}(x),N^*((B^{\sharp})^{-1}(y))]^{\pi^\sharp}\\
~ &&-N^*[(B^{\sharp})^{-1}(x),(B^{\sharp})^{-1}(y)]^{\pi^\sharp}\Big)\\
~ &\stackrel{\eqref{app}}=&(B^\sharp)^{-1}\Big([x,y]^{N\circ R}\Big)-\Big([(B^{\sharp})^{-1}N(x),(B^{\sharp})^{-1}(y)]^{\pi^\sharp}
+[(B^{\sharp})^{-1}(x),(B^{\sharp})^{-1}N(y)]^{\pi^\sharp}\\
~ &&-N^*[(B^{\sharp})^{-1}(x),(B^{\sharp})^{-1}(y)]^{\pi^\sharp}\Big)\\
~ &\stackrel{\eqref{con1}}=&(B^\sharp)^{-1}\Big([x,y]^{N\circ R}\Big)-\Big((B^{\sharp})^{-1}[Nx,y]^R+(B^{\sharp})^{-1}[x,Ny]^R-N^*(B^\sharp)^{-1}[x,y]^R\Big)\\
~ &=&(B^\sharp)^{-1}\Big([x,y]^{N\circ R}\Big)-(B^{\sharp})^{-1}\Big([Nx,y]^R+[x,Ny]^R-N[x,y]^R\Big)\\
~ &=&(B^\sharp)^{-1}\Big([x,y]^{N\circ R}-[x,y]^R_N\Big)=0,
\end{eqnarray*}
and
\begin{eqnarray*}
~ &&\Courant{\alpha,\beta,\gamma}^{N\circ \pi^\sharp}-\Courant{\alpha,\beta,\gamma}^{\pi^\sharp}_{N^*}\\
~ &=&\Courant{\alpha,\beta,\gamma}^{N\circ \pi^\sharp}
-\Big(\Courant{N^*(\alpha),N^*(\beta),\gamma}^{\pi^\sharp}+\Courant{N^*(\alpha),\beta,N^*(\gamma)}^{\pi^\sharp}
+\Courant{\alpha,N^*(\beta),N^*(\gamma)}^{\pi^\sharp}\\
~ &&-N^*\big(\Courant{N^*(\alpha),\beta,\gamma}^{\pi^\sharp}+\Courant{\alpha,N^*(\beta),\gamma}^{\pi^\sharp}
+\Courant{\alpha,\beta,N^*(\gamma)}^{\pi^\sharp}\big)+(N^*)^2\Courant{\alpha,\beta,\gamma}^{\pi^\sharp}\Big)\\
~ &\stackrel{\eqref{con2}}=&(B^\sharp)^{-1}\Courant{x,y,z}^{N\circ R}
-\Big(\Courant{(B^\sharp)^{-1}N(x),(B^\sharp)^{-1}N(y),(B^\sharp)^{-1}(z)}^{\pi^\sharp}\\
~ &&+\Courant{(B^\sharp)^{-1}N(x),(B^\sharp)^{-1}(y),(B^\sharp)^{-1}N(z)}^{\pi^\sharp}
+\Courant{(B^\sharp)^{-1}(x),(B^\sharp)^{-1}N(y),(B^\sharp)^{-1}N(z)}^{\pi^\sharp}\\
~ &&-N^*\big(\Courant{(B^\sharp)^{-1}N(x),(B^\sharp)^{-1}(y),(B^\sharp)^{-1}(z)}^{\pi^\sharp}
+\Courant{(B^\sharp)^{-1}(x),(B^\sharp)^{-1}N(y),(B^\sharp)^{-1}(z)}^{\pi^\sharp}\\
~ &&+\Courant{(B^\sharp)^{-1}(x),(B^\sharp)^{-1}(y),(B^\sharp)^{-1}N(z)}^{\pi^\sharp}\big)
+(N^*)^2\Courant{(B^\sharp)^{-1}(x),(B^\sharp)^{-1}(y),(B^\sharp)^{-1}(z)}^{\pi^\sharp}\Big)\\
~ &\stackrel{\eqref{con2}}=&(B^\sharp)^{-1}\Courant{x,y,z}^{N\circ R}-\Big((B^\sharp)^{-1}\Courant{N(x),N(y),z}^R+(B^\sharp)^{-1}\Courant{N(x),y,N(z)}^R+(B^\sharp)^{-1}\Courant{x,N(y),N(z)}^R\\
~&&-N^*\circ (B^\sharp)^{-1}\big(\Courant{N(x),y,z}^R+\Courant{x,N(y),z}^R+\Courant{x,y,N(z)}^R\big)+(N^*)^2(B^\sharp)^{-1}\Courant{x,y,z}^R\Big)\\
~ &=&(B^\sharp)^{-1}\Courant{x,y,z}^{N\circ R}-(B^\sharp)^{-1}\Big(\Courant{N(x),N(y),z}^R+\Courant{N(x),y,N(z)}^R+\Courant{x,N(y),N(z)}^R\\
~ &&-N\big(\Courant{N(x),y,z}^R+\Courant{x,N(y),z}^R+\Courant{x,y,N(z)}^R\big)+N^2\Courant{x,y,z}^R\Big)\\
~ &=&(B^\sharp)^{-1}\Courant{x,y,z}^{N\circ R}-(B^\sharp)^{-1}\Courant{x,y,z}^R_N=(B^\sharp)^{-1}\Big(\Courant{x,y,z}^{N\circ R}-\Courant{x,y,z}^R_N\Big)=0,
\end{eqnarray*}
which implies that
$$[\alpha,\beta]^{N\circ \pi^\sharp}=[\alpha,\beta]^{\pi^\sharp}_{N^*} \quad{\rm and}\quad \Courant{\alpha,\beta,\gamma}^{N\circ \pi^\sharp}=\Courant{\alpha,\beta,\gamma}^{\pi^\sharp}_{N^*}.$$
Thus $(\pi,N)$ is an $r$-matrix-Nijenhuis structure. This proves (i).

(ii) can be proved similarly.
\end{proof}

We give some examples of Rota-Baxter-Nijenhuis structures on \LYA s to end up this section.

\begin{ex}
Let $\g=C^\infty([0,1])$ endowed with the following operations
\begin{eqnarray*}
[f,g](x)&=&f(x)g'(x)-g(x)f'(x),\\
\Courant{f,g,h}(x)&=&f(x)g'(x)h'(x)-g(x)f'(x)h'(x)-h(x)\Big(f(x)g''(x)-g(x)f''(x)\Big),\quad \forall x\in [0,1],
\end{eqnarray*}
for all $ f,g,h\in \g$. Then $(\g,\br,,\ltp)$ forms a \LYA. For all $f\in C^\infty([0,1])$, the {\em integral operator} $R:\g\longrightarrow\g$ defined to be
$$R(f)(x):=\int_0^xf(t){\mathrm d}t,\quad\forall x\in [0,1]$$
is a Rota-Baxter operator on $\g$, and $N:\g\longrightarrow\g$ given by
$$N(f)(x):=\lambda f(x),\quad\forall x\in [0,1]$$
is a Nijenhuis operator for $\lambda\in \mathbb K$.
Then $(R,N)$ is a Rota-Baxter-Nijenhuis structure on $\g$.
\end{ex}

\begin{ex}
Let $(\g,\br,,\ltp)$ be a $2$-dimensional \LYA ~with a basis  $\{e_1,e_2\}$, where $(\br,,\ltp)$ is defined to be
$$[e_1,e_2]=e_1,\quad\Courant{e_1,e_2,e_2}=e_1.$$ Then $(R,N)$ is a Rota-Baxter-Nijenhuis structure on $\g$, where
\begin{eqnarray*}
R=\begin{pmatrix}
 0 & a \\
 0 & 0
 \end{pmatrix},
 \quad N=\begin{pmatrix}
 0 & \lambda \\
 0 & 0
 \end{pmatrix}.
\end{eqnarray*}
\end{ex}

\end{document}